\documentclass[12pt, a4paper]{article}
\usepackage[top=2cm, bottom=2cm, left=2cm, right=2cm]{geometry}
\usepackage{amsthm,amsmath}
\usepackage{harrymacros}
\usepackage{enumitem}
\usepackage{color}

\DeclareMathOperator*{\esssup}{ess\,sup}
\DeclareMathOperator{\vspan}{span}

\newcommand{\into}{\hookrightarrow}

\newtheorem{innercustomthm}{Theorem}
\newenvironment{customthm}[1]
 {\renewcommand\theinnercustomthm{#1}\innercustomthm}
 {\endinnercustomthm}

\begin{document}

\title{A spectral approach to quenched linear and higher-order response for partially hyperbolic dynamics}
\author{Harry Crimmins\footnote{School of Mathematics and Statistics, University of New South Wales, Sydney NSW 2052, Australia. Email: harry.crimmins@unsw.edu.au} and Yushi Nakano\footnote{Department of Mathematics, Tokai University, 4-1-1 Kitakaname, Hiratuka, Kanagawa, 259-1292, Japan. Email: yushi.nakano@tsc.u-tokai.ac.jp}}
\date{\today}
\maketitle

\pagestyle{myheadings}

\begin{abstract}
  For smooth random dynamical systems we consider the quenched linear and higher-order response of equivariant physical measures to perturbations of the random dynamics.
  We show that the spectral perturbation theory of Gou{\"e}zel, Keller, and Liverani \cite{keller1999stability, gouezel2006banach}, which has been applied to deterministic systems with great success, may be adapted to study random systems that possess good mixing properties.
  As a consequence, we obtain general linear and higher-order response results for random dynamical systems that we then apply to random Anosov diffeomorphisms and random U(1) extensions of expanding maps.  We emphasise that our results apply to random dynamical systems over a general ergodic base map, and are obtained without resorting to infinite dimensional multiplicative ergodic theory.
\end{abstract}

\section{Introduction}
In this paper we study quenched response theory for random dynamical systems (RDSs). The setup is as follows: take $M$ to be a $\mathcal{C}^\infty$ Riemannian manifold with Riemannian measure $m$,  $(\Omega, \mathcal{F}, \mathbb{P})$ to be a Lebesgue space, and for some $r \ge 1$ and each $\epsilon \in (-1,1)$ let $\mathcal{T}_\epsilon : \Omega \to \mathcal{C}^{r+1}(M,M)$ denote a one-parameter family of random maps with a `measurable' dependence on $\omega$. 
After fixing an invertible, $\mathbb{P}$-ergodic map $\sigma : \Omega \to \Omega$,  from each $\mathcal{T}_\epsilon$ we obtain random dynamical systems $(\mathcal{T}_\epsilon, \sigma )$ whose trajectories are random variables of the form
\begin{equation*}
  x, \mathcal{T}_{\epsilon,\omega} (x), \mathcal{T}^{(2)}_{\epsilon,\omega} (x), \dots, \mathcal{T}^{(n)}_{\epsilon,\omega} (x), \dots,
\end{equation*}
where $\mathcal{T}^{(n)}_{\epsilon,\omega}$ is short for the composition $\mathcal{T}_{\epsilon, \sigma^{n-1} \omega} \circ \cdots \circ \mathcal{T}_{\epsilon, \omega}$.
A family of probability measures $\{\mu_{\epsilon,\omega}\}_{\omega \in \Omega}$ on $M$ is said to be \emph{equivariant} for $( \mathcal{T}_{\epsilon}, \sigma)$ if $\mu_{\epsilon,\omega} \circ \mathcal{T}_{\epsilon,\omega}^{-1}= \mu_{\epsilon, \sigma \omega} $ for $\mathbb{P}$-a.e. $\omega$ (see Subsection \ref{ss:RDS} for a precise definition).  
When $\mathcal{T}_\epsilon$ possesses some (partial) hyperbolicity and good mixing properties one hopes to find a unique \emph{physical} equivariant family of probability measures\footnote{A $\mathcal{T}$-equivariant family of measure $\{\mu_{\omega}\}_{\omega \in \Omega}$ is physical with respect to $m$ if $n^{-1}\sum_{i=0}^{n-1} \delta_{\mathcal{T}_{\sigma^{-i}\omega^i}(x)} \to \mu$ for $x$ in a (possibly $\omega$-dependent) positive $m$-measure set with $\mathbb{P}$-probability 1.}, as such objects describe the $m$-a.e. realized statistical behaviour of the given RDS.
Quenched response theory is, broadly speaking, concerned with questions of the regularity of the map $\epsilon \mapsto \{\mu_{\epsilon, \omega}\}_{\omega \in \Omega}$ and, in particular, how this regularity is inherited from that of $\epsilon \mapsto \mathcal{T}_{\epsilon}$.  The one-parameter family of random maps $\epsilon \mapsto \mathcal{T}_{\epsilon}$ is said to exhibit quenched linear response if the measures $\{\mu_{\epsilon,\omega}\}_{\omega \in \Omega}$ vary differentiability with $\epsilon$ in an appropriate topology, with quenched higher-order (e.g. quadratic) response being defined analogously.

Linear and higher-order response theory for deterministic (i.e. non-random) systems is an established area of research, and there are a plethora of methods available for treating various systems (see \cite{baladi2014linear} for a good review).
Response theory has been developed for expanding maps in one and many dimensions \cite{baladi2014linear, baladi2018dynamical, sedro2018regularity}, intermittent systems \cite{baladi2016linear,bahsoun2016linear, korepanov2016linear}, Anosov diffeomorphisms \cite{ruelle1997differentiation, ruelle1998nonequilibrium, gouezel2006banach}, partially hyperbolic systems \cite{dolgopyat2004differentiability}, and piecewise expanding interval maps \cite{baladi2007susceptibility,baladi2008linear}.
The tools and techniques one may apply to deduce response results are likewise numerous: there are arguments based on structural stability \cite{ruelle1997differentiation},  standard pairs \cite{dolgopyat2004differentiability}, the implicit function theorem \cite{sedro2018regularity}, and on the spectral perturbation theory of Gou{\"e}zel, Keller and Liverani \cite{keller1999stability, gouezel2006banach} (and variants thereof, e.g. \cite{galatolo2020quadratic}).  

On the other hand,  the literature on quenched response theory for random dynamical systems is relatively small, and has only recently become an active research topic.  With a few notable exceptions, most results for random systems have focussed on the continuity of the equivariant random measure \cite{baladi1996random, baladi1997correlation, froyland2014stability,  gonzalez2018stability, nakano2016stochastic},  although some more generally apply to the continuity of the Oseledets splitting and Lyapunov exponents associated to the RDS's Perron-Frobenius operator cocycle \cite{bogenschutz2000stochastic, crimmins2019stability}.  
Quenched linear and higher-order response results are, to the best of our knowledge, limited to \cite{sedro2020regularity}, where quenched linear and higher-order response is proven for general RDSs of $\mathcal{C}^k$ uniformly expanding maps, and to \cite{dragivcevic2020statistical}, wherein quenched linear response is proven for RDSs of Anosov maps nearby a fixed Anosov map.
The relatively fewer results for response theory in the random case has been largely attributed to the difficulty in finding appropriate generalisations of the tools, techniques and constructions that have succeeded in the deterministic case. 
While the authors believe this sentiment is generally well-founded,  in this paper we find that for quenched linear and higher-order response problems it is possible to directly generalise an approach from the deterministic case to the random case with surprisingly little trouble.  
In particular,  by building on \cite{nakano2016stochastic} we will show that the application of Gou{\"e}zel--Keller--Liverani (GKL) spectral perturbation theory to response problems can be `lifted' to the random case, allowing one to deduce corresponding quenched response from deterministic response `for free'.

In the deterministic setting the application of GKL perturbation theory to response problems is part of the more general `functional analytic' approach to studying dynamical systems, which recasts the investigation of invariant measures and statistical properties of dynamical systems in functional analytic and operator theoretic terms.
The hero of this approach is the Perron-Frobenius operator, which for a non-singular\footnote{A map $T: M \to M$ is \emph{non-singular} with respect to $m$ if $m(A) = 0$ implies that $m(T^{-1}(A)) = 0$.} map $T \in \mathcal{C}^{r+1}(M, M)$ is denoted by $\LL_T$ and defined for $f \in L^1(m)$ by
\begin{equation*}
  (\LL_T f)(x) = \sum_{T(y) = x} \frac{f(y)}{\abs{\det D_y T}}.
\end{equation*}
The key observation is that the statistical properties of $T$ are often encoded in the spectral data of $\LL_T$ provided that one consider the operator on an appropriate Banach space \cite{baladi2000positive,baladi2018dynamical,liverani2004invariant, galatolo2015statistical}. 
Specifically,  one desires a Banach space for which $\LL_T$ is bounded and quasi-compact (in addition to some other benign conditions),  since in these cases a unique physical invariant measure $\mu_T$ for $T$ is often obtained as a fixed point of $\LL_T$. 
One may then attempt to answer response theory questions by studying the regularity of the map $T \mapsto \LL_T$ with a view towards deducing the regularity of $T \mapsto \mu_T$ via some spectral argument. The main obstruction to carrying out such a strategy is that $T \mapsto \LL_T$ is usually not continuous in the relevant operator norm, and so standard spectral perturbation theory (e.g. Kato \cite{kato1966perturbation}) cannot be applied.  
Instead, however,  one often has that $T \mapsto \LL_T$ is continuous (or $\mathcal{C}^k$) in some weaker topology,  and by applying GKL spectral perturbation theory it is then possible to deduce regularity results for $T \mapsto \mu_T$.

The main contribution of this paper is to show that the strategy detailed in the previous paragraph may still be applied in the random case to deduce quenched linear and higher-order response results.  More precisely,  with $\{(\mathcal{T}_\epsilon, \sigma)\}_{\epsilon \in (-1,1)}$ denoting the RDSs from earlier,  the main (psuedo) Theorem of this paper is the following (see Theorem \ref{thm:quenched_linear_response} for a precise statement and Section \ref{s:application} for our application to RDSs):
\begin{customthm}{A}\label{thm:A}
	Suppose that $(\mathcal{T}_{0}, \sigma)$ exhibits $\omega$-uniform exponential mixing on $M$,  and that for $\mathbb{P}$-a.e. $\omega$ the hypotheses of GKL perturbation theory are `uniformly' satisfied for the one-parameter families $\epsilon \mapsto \mathcal{T}_{\epsilon, \omega}$ as in the deterministic case.  
	Then whatever linear and higher-order response results that hold $\mathbb{P}$-a.e.  at $\epsilon = 0$  for the physical invariant probability measures of the one-parameter families $\epsilon \mapsto \mathcal{T}_{\epsilon, \omega}$ also hold in the quenched sense for the equivariant physical probability measures of the one-parameter family $ \epsilon \mapsto \{(\mathcal{T}_{\epsilon}, \sigma)\}_{\epsilon \in (-1,1)}$ of RDSs.
\end{customthm}
We note that despite the mixing requirement placed on $(\mathcal{T}_{0}, \sigma)$ in Theorem \ref{thm:A} we do not require that $\sigma$ exhibit any mixing behaviour, other than being ergodic.
The general strategy behind the proof of Theorem \ref{thm:A} is to consider for each $\epsilon \in (-1,1)$ a `lifted' operator obtained from the Perron-Frobenius operators $\{ \LL_{\mathcal{T}_{\epsilon, \omega}}\}_{\omega \in \Omega}$ associated to $\{ \mathcal{T}_{\epsilon, \omega}\}_{\omega \in \Omega}$.  
Then, using the fact that the hypotheses of the Gou{\"e}zel--Keller--Liverani Theorem (Theorem \ref{thm:GKL}) are satisfied `uniformly' for the Perron-Frobenius operators $\epsilon \mapsto \LL_{\mathcal{T}_{\epsilon, \omega}}$ and $\omega$ in some $\mathbb{P}$-full set, we deduce that the Gou{\"e}zel--Keller--Liverani Theorem may be applied to the lifted operator. By construction the fixed points of these lifted operators are exactly the equivariant physical probability measures of the corresponding RDS, and so we obtain the claimed linear and higher-order response via the conclusion of the Gou{\"e}zel--Keller--Liverani Theorem.
Using Theorem \ref{thm:A} we easily obtain new quenched linear and higher-order response results for random Anosov maps (Theorem \ref{thm:quenched_response_anosov}) and for random U(1) extensions of expanding maps (Theorem \ref{thm:1121}).  We note that our examples consist of random maps that are uniformly close to a fixed system.  
However,  this is not a strict requirement for the application of our theory and one could also consider `non-local' examples e.g.  it is clear that the arguments in Section \ref{s:application} are applicable to random systems consisting of arbitrary $\mathcal{C}^k$ expanding maps.

The structure of the paper is as follows. 
In Section \ref{sec:prelim} we introduce some conventions that are used throughout the paper and review some preliminary material related to random dynamical systems and the Gou{\"e}zel--Keller--Liverani Theorem. 
In Section \ref{sec:spectral_approach_stability} we consider random operator cocycles and their `lifts', and then prove our main abstract result, Theorem \ref{thm:quenched_linear_response},  which is a version of the Gou{\"e}zel--Keller--Liverani Theorem for the `lifts' of certain operator cocycles.
In Section \ref{s:application} we discuss how Theorem \ref{thm:quenched_linear_response} may be applied to study the quenched linear and higher-order response of general random $\mathcal{C}^{r+1}$ dynamical systems,  and then consider in detail the cases of random Anosov maps and random U(1) extensions of expanding maps. 
Lastly, Appendix \ref{a:pt} contains the proof of a technical lemma from Section \ref{s:application}.

\section{Preliminaries}\label{sec:prelim}

We adopt the following notational conventions:
\begin{enumerate}
  \item The symbol `$C$' will, unless otherwise stated, be used to indiscriminately refer to many constants, which are uniform (or almost surely uniform) and whose value may change between usages. If we wish to emphasise that $C$ depends on parameters $a_1, \dots, a_n$ we may write $C_{a_1, \dots, a_n}$ instead.
  \item If $X$ and $Y$ are topological vector spaces such that $X$ is continuously included into $Y$ then we will write $X \into Y$.
  \item If $X$ and $Y$ are Banach spaces then we denote the set of bounded, linear operators from $X$ to $Y$ by $L(X,Y)$.
  When $X=Y$, it is simply written as $L(X)$.
  \item When $X$ is a metric space we denote the Borel $\sigma$-algebra on $X$ by $\mathcal{B}_X$.
  \item If $A \in L(X)$ then we denote the spectrum of $A$ by $\sigma(A)$ and the spectral radius by $\rho(A)$. 
  We will frequently consider operators acting on a number of spaces simultaneously, and in such a situation we may denote $\sigma(A)$ and $\rho(A)$  by $\sigma(A | X)$ and $\rho(A | X)$, respectively,  for clarity.
\end{enumerate}

\subsection{Random dynamical systems}\label{ss:RDS}
Let  $(\Omega, \mathcal F , \mathbb{P})$ be a probability space and $\sigma:\Omega\to\Omega$  a \color{black} measurably invertible, \color{black} measure-preserving map.
For a measurable space $(\Sigma , \mathcal G)$,  we say that a measurable map $\Phi: \mathbb N_0 \times \Omega \times \Sigma \to \Sigma$ is a \emph{random dynamical system} (or, for short, an \emph{RDS}) on $\Sigma$ over the driving system $\sigma$ if
\[
\varphi ^{(0)} _\omega = \mathrm{id} _{\Sigma}  \quad\text{and}\quad \varphi ^{(n+m)} _ \omega  = \varphi ^{(n)}_{ \sigma ^m\omega }\circ  \varphi ^{(m)}_\omega
\]
for each $n, m \in \mathbb N_0$ and $\omega \in \Omega$, with the notation $\varphi ^{(n)}_\omega =\Phi (n,\omega ,\cdot )$ and $\sigma \omega =\sigma (\omega )$, where $\mathbb N_0 =\{0\} \cup \mathbb N$.
A standard reference for random dynamical systems is the monograph by Arnold \cite{Arnoldbook}. 
It is easy to check that 
\begin{equation}\label{eq:0220b2}
\varphi ^{(n)}_\omega = \varphi _{\sigma ^{n-1}\omega }\circ \varphi _{\sigma ^{n-2}\omega } \circ \cdots \circ \varphi _\omega 
\end{equation}
with the notation $\varphi _\omega = \Phi (1, \omega , \cdot )$.
Conversely, for each measurable map $\varphi : \Omega \times \Sigma\to \Sigma: (\omega , x) \mapsto \varphi _\omega (x)$, the measurable map $(n,\omega , x) \mapsto \varphi _\omega ^{(n)}(x)$ given by \eqref{eq:0220b2} is an RDS. 
We call it an \emph{RDS induced by $\varphi$ over $\sigma$}, and simply denote it by $(\varphi , \sigma )$.

It is easy to see that if we define a skew-product map $\Theta : \Omega \times \Sigma \to \Omega \times \Sigma$ by $\Theta (\omega , x)=(\sigma \omega , \varphi _\omega (x))$ for each  $(\omega , x)\in \Omega \times \Sigma$, then 
\[
\Theta ^n (\omega , x) =(\sigma ^n \omega , \varphi ^{(n)}_\omega (x)) \quad \text{for all $n\in \mathbb N_0$}.
\]
Furthermore, a probability measure $\mu$ on $(\Omega \times \Sigma , \mathcal F\times \mathcal G)$ is invariant for $\Theta$ (i.e.~$\mu \circ \Theta ^{-1}=\mu$) and $ \mu \circ \pi _\Omega ^{-1} = \mathbb P$, where $\pi _\Omega (\omega ,x)=\omega$ for each $(\omega , x)\in \Omega \times \Sigma$, if and only if there is a measurable family of probability measure $\{ \mu _\omega \} _{\omega \in \Omega}$ (i.e.~for each $A\in \mathcal G$, the map $\omega \mapsto \mu _\omega (A)$ is $(\mathcal F, \mathcal B_{\mathbb R})$-measurable) such that $\mu (A)= \int _\Omega \int _\Sigma 1_A(\omega ,x) \mu _\omega (\mathrm{d}x) \mathbb P(\mathrm{d}\omega )$ for each $A\in \mathcal F\times \mathcal G$, and that 
%
\begin{equation}\label{eq:1122a}
\mu_{\omega } \circ \varphi _{\omega }^{-1} = \mu_{\sigma \omega } \quad \text{ for almost every $\omega \in \Omega$}
\end{equation}
(cf.~\cite[Subsection 1.4]{Arnoldbook}).
Hence, 
 we say that 
 a measurable family of probability measure $\{ \mu _\omega \} _{\omega \in \Omega}$ is 
  \emph{equivariant}
   for $(\varphi , \sigma )$ if it satisfies \eqref{eq:1122a}.

\subsection{The Gou{\"e}zel--Keller--Liverani Theorem}
We recall the statement of the Gou{\"e}zel--Keller--Liverani theorem from \cite{baladi2018dynamical} (although we note that the result first appeared in full generality in \cite{gouezel2006banach,gouezelcorrigendum}, and in less generality in \cite{keller1999stability}).
Fix an integer $N \ge 1$ and let $E_j$, $j \in \{0, \dots, N\}$, be Banach spaces with $E_j \into E_{j-1}$ for each $j \in \{1, \dots, N\}$. 
For a family of linear operators $\{ A_\epsilon \}_{\epsilon \in [-1,1]}$ on these spaces we consider the following conditions:
\begin{enumerate}[label=\textnormal{(GKL\arabic*)}]
  \item \label{en:GKL1} For all $i \in \{1, \dots, N\}$ and $\abs{\epsilon} \le 1$ we have
  \begin{equation*}
    \norm{A_\epsilon}_{L(E_i)} \le C.
  \end{equation*}
  \item \label{en:GKL2} There exists $M > 0$ such that $\norm{A_\epsilon^n}_{L(E_0)} \le C M^n$ for all $\abs{\epsilon} \le 1$ and $n \in \N$.
  \item \label{en:GKL3} There exists $\alpha  < M$ such that for every  $\abs{\epsilon} \le 1$, $f \in E_1$ and $n \in \N$ we have
  \begin{equation*}
    \norm{A_\epsilon^n f}_{E_1} \le C\alpha ^n\norm{f}_{E_1} + C M^n \norm{f}_{E_0}.
  \end{equation*}
  \item \label{en:GKL4} For every $\abs{\epsilon} \leq 1$ we have
  \begin{equation*}
    \norm{A_\epsilon - A_0}_{L(E_N, E_{N-1})} \le C \abs{\epsilon}.
  \end{equation*}
\end{enumerate}
If $N \ge 2$ we have the following additional requirement:
\begin{enumerate}[label=\textnormal{(GKL5)}]
  \item \label{en:GKL5} There exist operators $Q_1, \dots,Q_{N-1}$ such that for all $j \in \{1, \dots, N-1\}$ and $i \in \{j, \dots, N\}$ we have
 \begin{equation}\label{eq:1203c}
    \norm{Q_j}_{L(E_i, E_{i-j})} \le C,
  \end{equation}
and that for all $\abs{\epsilon} \leq 1$ and $j \in \{2, \dots, N \}$ we have
 \begin{equation}\label{eq:1203d}
    \norm{A_\epsilon - A_0 - \sum_{k=1}^{j-1} \epsilon^k Q_k }_{L(E_N, E_{N-j})} \le C \abs{\epsilon}^j.
  \end{equation}
\end{enumerate}

\begin{theorem}[The Gou{\"e}zel--Keller--Liverani Theorem, {\cite[Theorem A.4]{baladi2018dynamical}}]\label{thm:GKL}
  Fix an integer $N \ge 1$ and let $E_j$, $j \in \{0, \dots, N\}$, be Banach spaces with $E_j \into E_{j-1}$ for each $j \in \{1, \dots, N\}$.
  Suppose that $\{ A_\epsilon \}_{\epsilon \in [-1,1]}$ satisfies \ref{en:GKL1}-\ref{en:GKL4} and if $N \ge 2$ then also \ref{en:GKL5}.
  For $z \notin \sigma(A_0 | E_N)$ set $R_0(z) = (z - A_0)^{-1}$ and define
  \begin{equation}\label{eq:GKL_1}
    S_\epsilon^{(N)}(z) = R_0(z) + \sum_{k=1}^{N-1} \epsilon^k  \sum_{j=1}^k \sum_{\substack{l_1 + \cdots + l_j = k \\ l_i \ge 1}} R_0(z) Q_{l_1} R_0(z) \cdots  R_0(z)Q_{l_j}R_0(z).
  \end{equation}
  In addition, for any $a > \alpha $ let
  \begin{equation*}
    \eta = \frac{\log(a/\alpha )}{\log(M/\alpha )},
  \end{equation*}
  and for $\delta > 0$ set
  \begin{equation*}
    \mathcal{V}_{\delta, a}(A_0) = \left\{z \in \C : \abs{z} \ge a \text{ and } \dist\left(z, \sigma\left(A_0 | E_j \right)\right) \ge \delta, \quad \forall j \in \{1, \dots, N\} \right\}.
  \end{equation*}
  There exist $\epsilon_0 > 0$ so that $\mathcal{V}_{\delta, a}(A_0) \cap \sigma(A_\epsilon | E_1 ) = \emptyset$ for every $\abs{\epsilon} \le \epsilon_0$ and so that, for each $z \in \mathcal{V}_{\delta, a}(A_0)$, we have
  \begin{equation*}
    \norm{(z- A_\epsilon)^{-1}}_{L(E_1)} \le C,
  \end{equation*}
  and
  \begin{equation*}
    \norm{(z- A_\epsilon)^{-1} - S_\epsilon^{(N)}(z)}_{L(E_N, E_0)} \le C \abs{\epsilon}^{N-1 + \eta}.
  \end{equation*}
\end{theorem}

\begin{remark}
 While the Gou{\"e}zel--Keller--Liverani Theorem as stated in Theorem \ref{thm:GKL} is true, there is an error in the proof of the result in both \cite{gouezel2010characterization} and \cite{baladi2018dynamical}. 
 For details of the error we refer the reader to \cite{gouezelcorrigendum},  and to the proof of \cite[Theorem 3.3]{gouezel2010characterization} for a corrected argument.
\end{remark}

\begin{remark}
	We emphasise that the inclusion $E_{j} \subset E_{j-1}$ need not be compact in Theorem \ref{thm:GKL}, which will be important in  our application in Section \ref{s:application}.	
\end{remark}

\section{A spectral approach to stability theory for operator cocycles}\label{sec:spectral_approach_stability}

Let $X$ be a Banach space, and $\mathcal{S}_{L(X)}$ denote the $\sigma$-algebra generated by the strong operator topology on $L(X)$.
   If $A : \Omega \to L(X)$ is $(\mathcal{F}, \mathcal{S}_{L(X)})$-measurable then we say it is \emph{strongly measurable}.
For an overview of the properties of strong measurable maps we refer the reader to \cite[Appendix A]{GTQuas1}.
The following lemma records the main properties of strongly measurable maps that we shall use.
\begin{lemma}[{\cite[Lemmas A.5 and A.6]{GTQuas1}}]
  \label{lemma:strong_measurability_summary}
 Suppose that $X$ is a separable Banach space and that $(\Omega, \mathcal{F}, \mathbb{P})$ is a Lebesgue space. Then
  \begin{enumerate}
    \item The set of strongly measurable maps is closed under (operator) composition i.e.~if $A_i : \Omega \to L(X)$, $i \in \{1, 2\}$, are strongly measurable then so to is $A_2 A_1 : \Omega \to L(X)$.
    \item If $A : \Omega \to L(X)$ is strongly measurable and $f : \Omega \to X$ is $(\mathcal{F}, \mathcal{B}_X)$-measurable then $\omega \mapsto A_\omega f_\omega$ is $(\mathcal{F}, \mathcal{B}_X)$-measurable too.
    \item If $A : \Omega \to L(X)$ is such that $\omega \mapsto A(\omega)f$ is $(\mathcal{F}, \mathcal{B}_X)$-measurable for every $f \in X$ then $A$ is strongly measurable.
  \end{enumerate}
\end{lemma}

As a slight abuse of notation, for a given strongly measurable map $A: \Omega \to L(X)$,  we  denote an $(\mathcal{F}\times  \mathcal{B}_X, \mathcal{B}_X)$-measurable map $ (\omega , f) \mapsto A(\omega )f$  by $A$.
In light of the previous lemma, 
we may now formally define the main objects of study for this section.


\begin{definition}\label{defn:linear_rds}
  An RDS $(A, \sigma )$ on $X$ induced by $A$ is called an \emph{operator cocycle} 
   (or a \emph{linear RDS})
    if $(\Omega, \mathcal{F}, \mathbb{P})$ is a Lebesgue space, $\sigma : \Omega \to \Omega$ is an invertible, ergodic, $\mathbb P$-preserving map,  $X$ is a separable Banach space and $A : \Omega  \mapsto L(X)$ is strongly measurable.
  We say that $(A, \sigma )$ is bounded if $A \in L^\infty(\Omega, L(X))$.
\end{definition}

Throughout the rest of this paper, we assume that $(\Omega, \mathcal{F}, \mathbb{P})$ is a Lebesgue space and  $\sigma : \Omega \to \Omega$ is an invertible, ergodic, $\mathbb P$-preserving map.
An operator cocycle $(A, \sigma )$ is explicitly written as  a measurable map
\[
\mathbb N_0 \times \Omega \times X \to X: (n, \omega ,f )\mapsto A^{(n)}(\omega )f, \quad A^{(n)}(\omega):=A(\sigma^{n-1} \omega) \circ \cdots \circ A(\omega).
\]
We denote by $X^*$ the dual space of $X$.
\begin{definition}\label{dfn:Markov}
  Let $\xi \in X^*$ be non-zero. We say that $A \in L(X)$ is $\xi$-Markov if  $\xi (A f) = \xi (f)$ for every $f \in X$.
  We say that an operator cocycle $(A,\sigma)$ is $\xi$-Markov if $A$ is almost surely $\xi$-Markov.
\end{definition}
 Notice that our terminology in Definition \ref{dfn:Markov} may be slightly non-standard: in the literature  
  a linear operator $A: X\to X$  is called  Markov   if $X=L^1(S, \mu )$ for a probability space $(S, \mu)$ and $A$ is positive (i.e.~$A f \geq 0$ $\mu$-almost everywhere if $f\geq 0$ $\mu$-almost everywhere) 
 and $\xi$-Markov  with $\xi (f) = \intf _S f d\mu$ (cf.~\cite{lasota2013chaos}).
 See also Definition \ref{dfn:1208} for more general definition of positivity.
We do not add the positivity condition to Definition \ref{dfn:Markov}   to make  clear that the result in this section holds without it.

\begin{definition}
  Suppose that $(A,\sigma)$ is a $\xi$-Markov operator cocycle for some non-zero $\xi \in X^*$.
  We say that $(A,\sigma)$ is $\xi$-mixing with rate $\rho \in [0,1)$ if for every $n \in \N$ we have

  \begin{equation}\label{eq:uniform_mixing}
    \esssup_{\omega \in \Omega} \norm{\restr{A^{(n)}(\omega)}{\ker \xi}} \le C \rho ^n.
  \end{equation}
 We say that $(A,\sigma)$ is $\xi$-mixing if it is $\xi$-mixing with some rate $\rho \in [0,1)$.
\end{definition}

Fix a non-zero $\xi \in X^*$.   We define $\mathcal{X} \equiv  \mathcal{X}_\xi$ as 
\begin{equation*}
  \mathcal{X}  = \{ f \in L^\infty(\Omega, X) : \xi(f) \text{ is almost surely constant} \}.
\end{equation*}
Since $\mathcal{X}$ is a closed subspace of $L^\infty(\Omega, X)$, it is a Banach space with the usual norm.
If $(A,\sigma)$ is a bounded $\xi$-Markov operator cocycle then we define $\mathbb{A} : \mathcal{X} \to \mathcal{X}$ by
\begin{equation}\label{eq:1122c}
  (\mathbb{A} f)(\omega) = A(\sigma^{-1}(\omega)) f(\sigma^{-1}(\omega)).
\end{equation}
We say that $\mathbb{A}$ is the \emph{lift} of $(A,\sigma)$. That $\mathbb{A} \in L(\mathcal{X})$ follows from Lemma \ref{lemma:strong_measurability_summary} and the boundedness of $(A,\sigma)$
\color{black}(see \cite{nakano2016stochastic} for a possible extension of the lift to the case when $\sigma$ is not invertible).\color{black} 
The following proposition  is a natural generalisation of \cite[Proposition 2.3]{nakano2016stochastic}.

\begin{proposition}\label{prop:spectral_properties_lift}
  Fix non-zero $\xi \in X^*$. 
  If $(A,\sigma)$ is a bounded, $\xi$-Markov, $\xi$-mixing operator cocycle
   with rate $\rho \in [0,1)$
    then $1$ is a simple eigenvalue of $\mathbb{A}$ and $\sigma(\mathbb{A}| \mathcal{X} ) \setminus \{1\} \subseteq \{ z \in \C : \abs{z} \le \rho \}$.
  \begin{proof}
 For each $c\in \mathbb C$, let
    \begin{equation}\label{eq:spectral_properties_lift_1}
      \mathcal{X} _{c}= \{ f \in \mathcal{X}  : \xi(f) = c \text{ almost surely}\}.
    \end{equation}
    We note that $ \mathcal{X} _{c}$ is non-empty since $\xi$ is assumed to be non-zero.
    Since $(A,\sigma)$ is a $\xi$-Markov operator cocycle, the lift $\mathbb{A}$ preserves $ \mathcal{X} _{c}$.
    For any $f,g \in  \mathcal{X} _{c}$ one has $ f-g \in \mathcal{X}_0 $ (i.e.~$f-g \in \ker \xi$ almost surely), and so as $(A,\sigma)$ is $\xi$-mixing with rate $\rho$ we have for every $n \in \N$ and almost every $\omega \in \Omega$ that
    \begin{equation}\begin{split}
      \norm{(\mathbb{A}^n f)(\omega) - (\mathbb{A}^n g)(\omega)} &= \norm{A^{(n)}(\sigma^{-n}(\omega))(f_{\sigma^{-n}(\omega)} - g_{\sigma^{-n}(\omega)})} \\
      &\le C \rho ^n \norm{f_{\sigma^{-n}(\omega)} - g_{\sigma^{-n}(\omega)}}.
    \end{split}\end{equation}
    Upon taking the essential supremum we see that $\mathbb{A}^n$ is a contraction mapping on $ \mathcal{X} _{c}$ for large enough $n$. 
    Since $ \mathcal{X} _{c}$ is complete, it follows that $\mathbb{A}$ has a unique fixed point $v_c$ in $ \mathcal{X} _{c}$. 
Obviously $v_c =c v_1$, and thus  
$1$ is an eigenvalue of $\mathbb{A}$ on $\mathcal{X} $. 
Furthermore, $ \mathcal{X} = \vspan\{v _1\} \oplus \mathcal{X}_{0}$
     (indeed,
     for every $f \in \mathcal{X} $ we can write $f = f_1 + f_0$ where $f_1 = \xi(f) v_1 \in \vspan\{v_1\}$ and $f_0 = f - f_1 \in \mathcal{X}_{0}$, and note that $\vspan\{v _1\}$ and $\mathcal{X}_{0}$ are closed subspaces).

    Since $\mathbb{A}$ preserves both $\vspan\{v _1\}$ and  $\mathcal{X}_{0}$, we have
    \begin{equation*}
      \sigma(\mathbb{A} | \mathcal{X} ) = \sigma(\mathbb{A} |\vspan\{v _1\}) \sqcup \sigma(\mathbb{A} |\mathcal{X}_{0}).
    \end{equation*}
    It is clear that $\sigma(\mathbb{A} |\vspan\{v _1 \})$ only consists of 
      a simple eigenvalue $1$, while $\rho(\mathbb{A} |\mathcal{X}_{0}) \le \rho $ since $(A,\sigma)$ is $\xi$-mixing with rate $\rho$.
    Thus  $\sigma(\mathbb{A} | \mathcal{X}  ) \setminus \{1\} = \sigma(\mathbb{A} |\mathcal{X}_{0}) \subseteq \{ z \in \C : \abs{z} \le \rho \}$.
  \end{proof}
\end{proposition}

\subsection{Main result}
Given a bounded, $\xi$-Markov, $\xi$-mixing operator cocycle $(A,\sigma)$ we are interested in question of stability (and differentiability) of the $\xi$-normalised fixed point $v$ of $\mathbb{A}$. To this end we formulate a number of conditions on operator cocycles that are reminiscent of the conditions of the Gou{\"e}zel--Keller--Liverani Theorem. 

Fix an integer $N \ge 1$ and let $E_j$, $j \in \{0, \dots, N\}$, be Banach spaces.
Let $\{(A_\epsilon , \sigma )\}_{\epsilon \in [-1,1]}$ be a family of operator cocycles on these spaces.
    \begin{enumerate}
    \item[(QR0)] \label{en:QR0} 
 $E_j \into E_{j-1}$ for each $j \in \{1, \dots, N\}$, $E_N$ is separable and  $\norm{\cdot}_{E_j}$-dense in $E_j$   for each $j \in \{0, \dots, N\}$ (in particular, $E_1$ is separable).
There exists non-zero  $\xi \in E_0^*$ such that $(A_\epsilon , \sigma )$ is $\xi$-Markov on $E_j$ for each $\vert \epsilon \vert \leq 1$, $j \in \{0, \ldots , N\}$ and that $(A_0 , \sigma )$ is $\xi$-mixing  on $E_j$ for each $j \in \{1, N\}$.
 \item[(QR1)]  \label{en:QR1} For all $i \in \{1, \dots, N\}$ and $\abs{\epsilon} \le 1$ we have $\esssup_{\omega} \norm{A_\epsilon(\omega)}_{L(E_i)} \le C$.
 \item[(QR2)]  \label{en:QR2} There exists $M > 0$ such that $  \esssup_\omega \norm{A_\epsilon^{(n)}(\omega)}_{L(E_0)} \le C M^n$
  for all $\abs{\epsilon} \le 1$ and $n \in \N$.
 \item[(QR3)]  \label{en:QR3} There exists $\alpha < M$ such that for every $f \in E_1$, $\abs{\epsilon} \le 1$ and $n \in \N$ we have
  \begin{equation*}
    \esssup_{\omega} \norm{A_\epsilon^{(n)}(\omega) f}_{E_1} \le C\alpha ^n\norm{f}_{E_1} + C M^n \norm{f}_{E_0}.
  \end{equation*}
 \item[(QR4)]  \label{en:QR4} For every $\abs{\epsilon} \leq 1$ we have
  \begin{equation*}
    \esssup_\omega \norm{A_\epsilon(\omega) - A_0(\omega)}_{L(E_N, E_{N-1})} \le C \abs{\epsilon}.
  \end{equation*}
  \end{enumerate}
  If $N \ge 2$ we have the following additional requirement:
  \begin{enumerate}[label=\textnormal{(QR5)}]
  \item \label{en:QR5} There exist operators $Q_1(\omega), \dots,Q_{N-1}(\omega)$ for each $\omega$ 
   such that for all $j \in \{1, \dots, N-1\}$ and $i \in \{j, \dots, N\}$ we have
  \begin{equation*}
       \esssup_{\omega} \norm{Q_j(\omega)}_{L(E_i, E_{i-j})} \le C,
  \end{equation*}
  and such that for all $\abs{\epsilon} \leq 1$ and $j \in \{2, \dots, N \}$ we have
  \begin{equation*}
    \esssup_{\omega}   
    \norm{A_\epsilon(\omega) - A_0(\omega) - \sum_{k=1}^{j-1} \epsilon^k Q_k(\omega) }_{L(E_N, E_{N-j})} \le C \abs{\epsilon}^j.
  \end{equation*}
  \end{enumerate}
We need not  assume 
 that 
 $Q_1, \ldots , Q_{N-1}$ are measurable\footnote{Recall that  the essential supremum 
 of a (not necessarily measurable)   complex-valued function $f$ on $\Omega$ is given as the infimum of $\sup _{\omega \in \Omega _0} \vert f(\omega )\vert $ over all $\mathbb P$-full measure sets $\Omega _0$.},
  which will make our application in Section \ref{s:application} simpler.

Our main theorem for this section is the following.

\begin{theorem}\label{thm:quenched_linear_response}
  Fix an integer $N \ge 1$, and let $E_j$, $j \in \{0, \dots, N\}$, be Banach spaces
    and $\{ (A_\epsilon ,\sigma ) \}_{\epsilon \in [-1,1]}$ a family of operator cocycles on these spaces. 
   Suppose that $\{ (A_\epsilon ,\sigma ) \}_{\epsilon \in [-1,1]}$ satisfies $\mathrm{(QR0)}$-$\mathrm{(QR4)}$
    and if $N \ge 2$ then also \ref{en:QR5}.
  Then there exists $\epsilon_0 \in (0, 1]$ such that one can find a unique $v_\epsilon \in L^\infty (\Omega , E_1)$ 
  such 
   that $A_\epsilon (\omega )v_\epsilon (\omega )=v_\epsilon (\sigma \omega )$ and $\xi (v_\epsilon (\omega ))=1$ for 
each $\epsilon \in (-\epsilon_0 , \epsilon_0)$ and almost every $\omega \in \Omega$, 
 that  
  \begin{equation*}
    \sup_{\abs{\epsilon} < \epsilon_0} \esssup_\omega \norm{v_\epsilon(\omega)}_{E_1} < \infty,
  \end{equation*}
 and 
  that $(A_\epsilon ,\sigma )$
  is $\xi$-mixing whenever $\abs{\epsilon} < \epsilon_0$.
  Moreover,   there exists $\{v_0^{(k)}\}_{k=1}^{N-1} \subset  L^\infty(\Omega, E_0) $ such that $\xi(v_0^{(k)}) = 0$ almost surely for each $k$, 
   and  that for every $\eta \in (0, \log (1/\alpha )/\log (M/\alpha ))$,
 we have
  \begin{equation}\label{eq:quenched_linear_response_1}
    v_\epsilon = v_0 + \sum_{k=1}^{N-1} \epsilon^k v_0^{(k)} + O_\eta(\epsilon^{N-1+\eta}),
  \end{equation}
  where $O_\eta(\epsilon^{N-1+\eta})$ is to be understood as an essentially bounded term in $E_0$ that possibly depends on $\eta$.
\end{theorem}

\begin{remark}
  One is free to take $E_0 = E_1 = \cdots = E_N$ in Theorem \ref{thm:quenched_linear_response}, in which case the conditions $\mathrm{(QR0)}$-$\mathrm{(QR3)}$
   collapse into a single bound and $\mathrm{(QR4)}$-$\mathrm{(QR5)}$
    become standard operator norm inequalities.
  Hence in this simple case one recovers an expected Banach space perturbation result.
\end{remark}

\begin{remark}
  We note that Theorem \ref{thm:quenched_linear_response} has been proven before for the cases where $N = 1$ and $N = 2$ in \cite{froyland2014stability} and \cite{dragivcevic2020statistical}, respectively.
\end{remark}

\begin{remark}\label{remark:mixing}
  The claim that `there exists $\epsilon_0 \in (0, 1]$ such that $(A_\epsilon ,\sigma )$ is $\xi$-mixing whenever $\abs{\epsilon} < \epsilon_0$' is exactly the content of \cite[Proposition 1]{dragivcevic2020statistical} (as well as being an easy corollary of \cite[Proposition 3.11]{crimmins2019stability}).
  Upon examining these proofs it is clear that something slightly stronger is true: in the setting of Theorem \ref{thm:quenched_linear_response}, for every $\kappa \in (\rho ,1)$ there exists $\epsilon_\kappa > 0$ such that for all $\epsilon \in (-\epsilon_\kappa, \epsilon_\kappa)$ one has
  \begin{equation}\label{eq:uniform_mixing_perturb}
    \sup_{n \in \N} \kappa^{-n} \esssup_{\omega} \norm{\restr{A_\epsilon^n}{\ker \xi}}_{L(E_1)} \le C.
  \end{equation}
\end{remark}

%

\subsection{The proof of Theorem \ref{thm:quenched_linear_response}}

Before detailing the proof of Theorem \ref{thm:quenched_linear_response} we introduce some basic constructs.
For each $j \in \{0, \dots, N\}$ let
\begin{equation*}
  \mathcal{E}_j = \{ f \in L^\infty(\Omega, E_j) : \xi(f) \text{ is almost surely constant} \}.
\end{equation*}
Since $\xi \in E_j^*$ for each $j \in \{0, \dots, N\}$ we observe that each $\mathcal{E}_j$ is a closed subspace of $L^\infty(\Omega, E_j)$, and consequently a Banach space.
Moreover, we have $\mathcal{E}_j \into \mathcal{E}_{j-1}$ for $j \in \{1, \dots, N\}$.
For each $j \in \{1, \dots, N\}$ we may consider the lift $\mathbb{A}_{\epsilon,j}$ of the operator cocycle $( A_\epsilon , \sigma )$ on $E_j$, although we will always omit the subscript $j$ and just write $\mathbb{A}_{\epsilon}$, which will be of no consequence.

The beginning of the proof of Theorem \ref{thm:quenched_linear_response} is straightforward.
First we note that (QR1) implies that $(A_0,\sigma )$  is bounded on $E_j$ for $j \in \{1, N\}$ and so Proposition \ref{prop:spectral_properties_lift} may be applied 
  to characterise the spectrum of $\mathbb{A}_0$ on $\mathcal{E}_1$ and $\mathcal{E}_N$.
  Let $\rho$ be the rate of $\xi$-mixing in (QR0), that is, $(A_0, \sigma )$ is $\xi$-mixing on $E_j$ with rate $\rho$ for each $j\in \{ 1, N\}$.
Then it follows from (QR0) and Proposition \ref{prop:spectral_properties_lift} that 1 is a simple eigenvalue of $\mathbb{A}_0$ when considered on either space, and we have
\begin{equation}\label{eq:quenched_linear_response_2}
  \sigma(\mathbb{A}_0 | \mathcal{E}_j) \setminus \{1\} \subseteq \{ z \in \C : \abs{z} \le \rho \}
\end{equation}
for $j \in \{1, N\}$.
For each $j \in \{1, \dots, N\}$ one may use basic functional analysis and the fact that $\mathcal{E}_N \into \mathcal{E}_j \into \mathcal{E}_1$ to deduce that 1 is a simple eigenvalue of $\mathbb{A}_0 : \mathcal{E}_j \to \mathcal{E}_j$ and that \eqref{eq:quenched_linear_response_2} holds.
As a consequence we find a $\xi$-normalised $v_0 \in \mathcal{E}_N$ that is the unique $\xi$-normalised fixed point of $\mathbb A_0 : \mathcal E_j \to \mathcal E_j$ for each $j \in \{1, \dots, N\}$.

We now turn to constructing the $\xi$-normalised fixed points of $\mathbb  A_\epsilon : \mathcal E_1 \to \mathcal E_1$.
By Remark \ref{remark:mixing} we may find some $\kappa \in (\rho ,1)$ and $\epsilon_0 > 0$ such that $(A_\epsilon ,\sigma )$ is $\xi$-mixing on $E_1$ with rate $\kappa$ for every $\epsilon \in (-\epsilon_0, \epsilon_0)$.
We note that each $( A_\epsilon ,\sigma )$ is bounded on $E_1$ due to (QR1), and so by Proposition \ref{prop:spectral_properties_lift} we find that $1$ is a simple eigenvalue of $\mathbb{A}_\epsilon : \mathcal{E}_1 \to \mathcal{E}_1$, and that $\sigma(\mathbb{A}_\epsilon | \mathcal{E}_1) \setminus \{1\} \subseteq \{z \in \C : \abs{z} \le \kappa\}$.
Thus $\mathbb  A_\epsilon : \mathcal E_1 \to \mathcal E_1$ has a unique $\xi$-normalised fixed point $v_\epsilon \in \mathcal{E}_1$ for each $\epsilon \in (-\epsilon_0, \epsilon_0)$ by Proposition \ref{prop:spectral_properties_lift}.
Moreover, by virtue of the uniform bound \eqref{eq:uniform_mixing_perturb} we may strengthen the conclusion of Proposition \ref{prop:spectral_properties_lift}: for all $n$ sufficiently large the family of maps $\{\mathbb{A}_\epsilon^n\}_{\abs{\epsilon} < \epsilon_0}$ uniformly contract the set $\mathcal{X} _1$ from \eqref{eq:spectral_properties_lift_1}.
Hence we deduce the bound
\begin{equation*}
  \sup_{\abs{\epsilon} < \epsilon_0} \esssup_\omega \norm{v_\epsilon(\omega)}_{E_1} < \infty,
\end{equation*}
as required for Theorem \ref{thm:quenched_linear_response}.

Thus to complete the proof of Theorem \ref{thm:quenched_linear_response} it suffices to prove \eqref{eq:quenched_linear_response_1}.
It may be easily seen from the proof of Proposition \ref{prop:spectral_properties_lift} that the eigenprojection $\Pi_\epsilon \in L(\mathcal{E}_1)$ of $\mathbb{A}_\epsilon : \mathcal{E}_1 \to \mathcal{E}_1$ onto the eigenspace for $1$ is defined for $f \in \mathcal{E}_1$ and $\epsilon \in (-\epsilon_0, \epsilon_0)$ by
\begin{equation*}
  \Pi_\epsilon (f) = \xi(f) v_\epsilon.
\end{equation*}
Since each $v_\epsilon$ is $\xi$-normalised, we consequently have
\begin{equation}\label{eq:quenched_linear_response_3}
  v_\epsilon = v_0 + (\Pi_\epsilon- \Pi_0)v_0.
\end{equation}
If $\delta \in (0, 1-\kappa )$ then $D_\delta = \{ z \in \C : \abs{z -1} = \delta\} \subseteq \C \setminus \sigma(\mathbb{A}_\epsilon | \mathcal{E}_1)$ for every $\epsilon \in (-\epsilon_0, \epsilon_0)$. Thus
\begin{equation}\label{eq:quenched_linear_response_4}
  \Pi_\epsilon = \intf_{D_\delta} (z - \mathbb{A}_\epsilon)^{-1} dz.
\end{equation}
Applying \eqref{eq:quenched_linear_response_4} to \eqref{eq:quenched_linear_response_3} yields
\begin{equation}\label{eq:quenched_linear_response_5}
  v_\epsilon = v_0 + \intf_{D_\delta} \left((z - \mathbb{A}_\epsilon)^{-1} - (z - \mathbb{A}_0)^{-1}\right)v_0 dz.
\end{equation}
The idea is to apply the Gou{\"e}zel--Keller--Liverani Theorem to the lifts $\{\mathbb{A}_{\epsilon}\}_{\epsilon \in [-1,1]}$ with Banach spaces $\{\mathcal{E}_j\}_{0 \le j \le N}$, and then develop a Taylor expansion in \eqref{eq:quenched_linear_response_5}.
The hypothesis that (QR1)-(QR4) holds for $\{ (A_\epsilon ,\sigma )\}_{\epsilon \in [-1, 1]}$ with constants almost surely independent of $\omega$ readily implies that the lifts $\{\mathbb{A}_{\epsilon}\}_{\epsilon \in [-1,1]}$ satisfy \ref{en:GKL1}-\ref{en:GKL4} for the spaces $\{\mathcal{E}_j\}_{0 \le j \le N}$.
Hence, in the case where $N = 1$ we may apply Theorem \ref{thm:GKL} to the lifts $\{\mathbb A_\epsilon \}_{\epsilon \in [-1, 1]}$.

However, the case where $N \ge 2$ is more delicate because the measurability of  $Q_j$ is not required in Theorem \ref{thm:quenched_linear_response}.
Thus, we introduce the following   functional space instead of $L^\infty (\Omega, E_j)$, where the objects are defined up to almost everywhere equality but we loosen the measurability requirement.
For each $j\in \{0,\ldots ,N\}$, let 
$B(\Omega , E_j)$ denote the set of (not necessarily measurable) bounded $E_j$-valued functions on $\Omega$ equipped with the uniform norm $\Vert f\Vert _{B(\Omega , E_j)} = \sup _{\omega \in \Omega} \Vert f(\omega )\Vert _{E_j}$ for each $f\in B(\Omega , E_j)$, and let
\[
\mathcal{N}_j = \left\{ f \in B(\Omega, E_j) : \text{
 $ f= 0 $ almost surely}
 \right\}.
\]
Then it is straightforward to see that $\mathcal{N}_j$ is a closed subspace of $B(\Omega, E_j)$,
 and thus we can form a quotient space 
 \[
  I^\infty( \Omega , E_j) = B(\Omega, E_j) / \mathcal{N}_j.
\]
Since $B(\Omega, E_j)$ is a Banach space, $I^\infty( \Omega , E_j) $ is also  a Banach space with respect to the quotient norm 
\[
\Vert f \Vert _{I^\infty( \Omega , E_j) } = \inf _{h\in \mathcal N_j} \Vert g - h \Vert _{B(\Omega , E_j )} , \quad f\in I^\infty( \Omega , E_j),
\]
where $g$ is a representative of $f$.
As for $L^\infty( \Omega , E_j)$,   under the identification of each element of $I^\infty( \Omega , E_j)$ with its representative,  we have
\[
\Vert f \Vert _{I^\infty( \Omega , E_j) } = \esssup _{\omega} \Vert f(\omega )\Vert _{E_j}.
\]
Thus under the identification we have $\Vert f\Vert _{L^\infty(\Omega, E_j)}  = \Vert f\Vert _{I^\infty(\Omega, E_j)}$ for each $f\in L^\infty(\Omega, E_j)$, that is, $L^\infty(\Omega, E_j)$ isometrically injects into $I^\infty(\Omega, E_j)$.
Finally let 
\begin{equation}\label{eq:gx1}
  \widetilde{\mathcal{E}}_j = \left\{ f \in I^\infty(\Omega, E_j) : \xi(f) \text{ is almost surely constant}\right\}.
\end{equation}
We simply write $\Vert f\Vert _{  \widetilde{\mathcal{E}}_j }$ for $\Vert f\Vert _{I^\infty (\Omega , E_j)}$ if  $f \in  \widetilde{\mathcal{E}}_j $. 
Repeating the previous argument, one can show  \ref{en:GKL1}-\ref{en:GKL4} for  the lifts $\{\mathbb{A}_{\epsilon}\}_{\epsilon \in [-1,1]}$ with respect to the spaces $\{  \widetilde{\mathcal{E}}_j\}_{0 \le j \le N}$.

We need some auxiliary lemmas. 
The first lemma is a standard exercise in functional analysis, which will allow us in Lemma \ref{lemma:measurable_completion} to characterise the relationship between the spaces $\widetilde{\mathcal{E}}_j$, $j \in \{0, \dots, N\}$.

\begin{lemma}\label{lemma:completion_of_kernel}
  Assume the setting of Theorem \ref{thm:quenched_linear_response}. For each $K \in \C$ and $j \in \{0, \dots, N\}$ let $U_{K,j} = \{ g \in E_j : \xi(g) = K \}$.
  Then there is a countable subset of $U_{K,j} \cap E_N$ that is $\norm{\cdot}_j$-dense in $U_{K,j}$.
  \end{lemma}
  \begin{proof}
    Fix $j$. We begin by reducing to the case where $K = 0$. Since $\xi$ is non-zero on $E_0$ and $E_N$ is $\norm{\cdot}_{E_0}$-dense in $E_0$ there exists some $g \in E_N$ such that $\xi(g) \ne 0$. Without loss of generality we may assume that $\xi(g) = 1$, in which case $U_{K,j} = Kg + U_{0,j}$.
    Thus to obtain a countable, $\norm{\cdot}_j$-dense subset of $U_{K,j}$ inside $U_{K,j} \cap E_N$ it suffices to do so for the case where $K = 0$ and then translate by $Kg$.
    We will now prove the lemma for $K = 0$.
    Since $E_N$ is separable and $\xi \in E_N^*$ there is a countable, $\norm{\cdot}_{E_N}$-dense subset $G$ of $U_{0,N} \cap E_N$. We will show that $G$ is $\norm{\cdot}_{E_j}$-dense in $U_{0,j}$ too.
    Let $u \in U_{0,j}$. As $E_N$ is $\norm{\cdot}_{E_j}$-dense in $E_j$ there exists a sequence $\{u_k\}_{k \in \N} \subseteq E_N$ with $\norm{\cdot}_{E_j}$-limit $u$.
    It follows that $\{u_k - \xi(u_k)g\}_{k \in \N}$ is contained in $U_{0,N}$ and satisfies $\lim_{k \to \infty} u_k - \xi(u_k)g = \lim_{k \to \infty} u_k = u$ in $E_j$ since $\xi(u_k) \to 0$.
    Lastly, $G$ is $\norm{\cdot}_{E_j}$-dense in $U_{0,N}$ and so for each $k$ we may find a $v_k \in G$ so that $\norm{v_k - (u_k - \xi(u_k)g)}_{E_N} \to 0$, which implies that $\lim_{k \to \infty} v_k = u$ in $E_j$.
    Thus the $\norm{\cdot}_{E_j}$-completion of $G$ is $U_{0,j}$.
  \end{proof}

\begin{lemma}\label{lemma:measurable_completion}
  Assume the setting of Theorem \ref{thm:quenched_linear_response}. For each $j \in \{1, \dots, N\}$ the space $\widetilde{\mathcal{E}}_j$ is equal to the $\norm{\cdot}_{\widetilde{\mathcal{E}}_j}$-completion of $\widetilde{\mathcal{E}}_N$.
  \end{lemma}
  \begin{proof}
    Fix $j \in \{1, \dots, N\}$ and suppose that $f \in \widetilde{\mathcal{E}}_j$.
    Recall that $\xi(f)$ is almost surely constant, and let $G \subseteq E_N$ denote a countable, $\norm{\cdot}_{E_j}$-dense subset of $\{g \in E_j : \xi(g) = \xi(f) \}$ as produced by Lemma \ref{lemma:completion_of_kernel}.
    Fix an enumeration $\{g_i\}_{i \in \N}$ of $G$ and for each $i,n \in \N$ set
    \begin{equation*}
      V_{i,n} =
      \begin{cases}
        B_{\norm{\cdot}_j}(n^{-1}, g_i) & i = 1, \\
        B_{\norm{\cdot}_j}(n^{-1}, g_i) \setminus \left(\bigcup_{k < i} B_{\norm{\cdot}_j}(n^{-1}, g_k)\right) & i > 1. \\
      \end{cases}
    \end{equation*}
    We note that $\{V_{i,n}\}_{i \in \N}$ is a countable partition of $E_j$ for each $n \in \N$, and that each $V_{i,n}$ is measurable in the Borel $\sigma$-algebra on $E_j$.
    We define a sequence of approximations $\{f_n\}_{n \in \N} \subseteq \widetilde{\mathcal{E}}_N$ by
    \begin{equation*}
      f_n(\omega) = \sum_{i=1}^\infty g_i \chi_{f^{-1}(V_{i,n})}(\omega).
    \end{equation*}
    Notice that 
     $\xi(f_n) = \xi(f)$ almost everywhere and that
    \begin{equation*}
      \esssup_{\omega} \norm{f_n(\omega) - f(\omega)}_{E_j} \le n^{-1}.
    \end{equation*}
    Hence $\lim_{n \to \infty} \norm{f_n - f}_{\widetilde{\mathcal{E}}_j} = 0$, and so $f$ is in the $\norm{\cdot}_{\widetilde{\mathcal{E}}_j}$-completion of $\widetilde{\mathcal{E}}_N$. The required claim immediately follows.
  \end{proof}

With Lemma \ref{lemma:measurable_completion} in hand we can deduce \ref{en:GKL5} for the lifted systems.
\begin{proposition}\label{prop:measurable_derivatives}
  Assume the setting of Theorem \ref{thm:quenched_linear_response} with $N \ge 2$. Then \ref{en:GKL5} holds with operators $\mathbb{Q}_j$ defined by
  \begin{equation*}
    (\mathbb{Q}_jf)(\omega) = Q_j(\sigma^{-1}\omega) f_{\sigma^{-1} \omega},
  \end{equation*}
  where $j \in \{1, \dots , N-1\}$, $i \in \{j, \dots, N\}$ and $f \in \widetilde{\mathcal{E}}_i$.
  \end{proposition}
  \begin{proof}
  It is straightforward to verify the inequalities in \ref{en:GKL5} from (QR5)
if 
we can show that 
    $\mathbb{Q}_j(\widetilde{\mathcal{E}}_i) \subset \widetilde{\mathcal{E}}_{i-j}$ for each $i\in \{j, \ldots , N\}$ and $j\in \{i,\ldots , N-1\}$ because $( A_\epsilon ,\sigma )$ is $\xi$-Markov for every $\epsilon \in [-1,1]$. 
%
%
Furthermore, for each 
 $f \in \widetilde{\mathcal{E}}_i$ we have
 \[
      \esssup_{\omega} \norm{(\mathbb{Q}_j f)(\omega)}_{E_{i-j}} \le \left( \esssup_{\omega} \norm{Q_j(\omega)}_{L(E_{i},E_{i-j})} \right)       \esssup_{\omega}\norm{f(\omega )}_{E_i},
 \]
    which is finite due to (QR5).
%
%
%
%
%
Hence it suffices to show that $\xi(\mathbb{Q}_jf)$ is almost surely constant for each $i\in\{1,\ldots ,N\}$, $j\in \{i,\ldots ,N-1\}$ and $f \in \widetilde{\mathcal{E}}_i$.
    By Lemma \ref{lemma:measurable_completion} we may find a sequence $\{f_n\}_{n \in \N} \subseteq \widetilde{\mathcal{E}}_N$ such that 
    $\norm{f_n - f}_{\widetilde{\mathcal{E}}_{i}} \to 0$ as $n\to\infty$.
   Using  \ref{en:QR5}, we have for each $n \in \N$ that
\[
      \mathbb{Q}_j f_n = \lim_{\epsilon \to 0} \left(\epsilon^{-j} (\mathbb{A}_\epsilon - \mathbb{A}_0)f_n - \sum_{k=1}^{j-1} \epsilon^{k-j} \mathbb{Q}_k f_n\right)
\]
   in $E_{N-j-1}$ almost surely,  with the convention $\sum_{k=1}^{0} \epsilon^{k-j} \mathbb{Q}_k f_n =0$.
This  implies that
\begin{equation}\label{eq:gw2}
      \mathbb{Q}_j f= \lim _{n\to \infty}\lim_{\epsilon \to 0} \left(\epsilon^{-j} (\mathbb{A}_\epsilon - \mathbb{A}_0)f_n - \sum_{k=1}^{j-1} \epsilon^{k-j} \mathbb{Q}_k f_n\right)
\end{equation}
       in $E_0$ almost surely.
Therefore,     since  $( A_\epsilon ,\sigma )$ is $\xi$-Markov for every $\epsilon \in [-1,1]$ we have almost surely that $\xi (\mathbb{Q}_1f) = 0$  and
    \begin{equation*}
      \xi (\mathbb{Q}_jf) = -\lim_{n \to \infty} \lim_{\epsilon \to 0} \sum_{k=1}^{j-1} \epsilon^{k-j} \xi (\mathbb{Q}_k f_n).
    \end{equation*}
 By induction with respect to $j$ we deduce that $\xi (\mathbb{Q}_jf) = 0$ almost surely.
  \end{proof}

By Proposition \ref{prop:measurable_derivatives} we have \ref{en:GKL5} for the lifts $\{\mathbb{A}_\epsilon\}_{\epsilon \in [-1,1]}$ with the spaces $\{ \widetilde{\mathcal{E}}_j) \} _{0\leq j\leq N}$ whenever $N \ge 2$ in the setting of Theorem \ref{thm:quenched_linear_response}, and so we can apply Theorem \ref{thm:GKL} in this case.
As a consequence we may now finish the proof of Theorem \ref{thm:quenched_linear_response}.
Let $\eta \in (0, \log (1/\alpha )/\log (M/\alpha ))$ and fix $a \in ( \alpha ,1)$ so that $\eta = \log (a/\alpha )/\log (M/\alpha )$.
Recall $\delta$ from \eqref{eq:quenched_linear_response_4} and notice that we may take $\delta$ to be as small as we like. 
Henceforth we fix $\delta \in (0, 1-a)$ and choose some $\delta_0 \in (0, \min \{\delta, 1 - a - \delta\})$.
Upon recalling the statement of Theorem \ref{thm:GKL} and our earlier characterisation of $\sigma(\mathbb{A}_0 | \widetilde{\mathcal{E}}_j)$ for $j \in \{1, \dots, N\}$ (recall the remark following \eqref{eq:quenched_linear_response_2}), we have
\begin{equation*}
  D_\delta \subseteq \{ z \in \C : \abs{z} \ge s \text{ and } \abs{z-1} \ge \delta_0 \} \subseteq \mathcal{V}_{\delta_0, a}(\mathbb{A}_0).
\end{equation*}
We now apply Theorem \ref{thm:GKL} to the lifts $\{\mathbb{A}_\epsilon\}_{\epsilon \in [-1,1]}$ with Banach spaces $\widetilde{\mathcal{E}}_j$, $j \in \{0, \dots, N\}$, to deduce the existence of $\epsilon_\eta \in (0, \epsilon_0)$ such that for every $\epsilon \in (-\epsilon_\eta, \epsilon_\eta)$ we have $\mathcal{V}_{\delta_0, a}(A_0) \cap \sigma(\mathbb{A}_\epsilon | \widetilde{\mathcal{E}}_1) = \emptyset$ and, for each $z \in \mathcal{V}_{\delta_0, a}(\mathbb{A}_0)$, that
\begin{equation}\label{eq:quenched_linear_response_6}
  \norm{(z- \mathbb{A}_\epsilon)^{-1} - \mathbb{S}_\epsilon^{(N)}(z)}_{L(\widetilde{\mathcal{E}}_N, \widetilde{\mathcal{E}}_0)} \le C \abs{\epsilon}^{N-1 + \eta},
\end{equation}
where $\mathbb{S}_\epsilon^{(N)}(z)$ is defined as in \eqref{eq:GKL_1}.
With \eqref{eq:quenched_linear_response_6} in hand we may proceed with obtaining \eqref{eq:quenched_linear_response_1} via \eqref{eq:quenched_linear_response_5}.
In particular, for $z \in D_\delta$ we have
\begin{equation}\begin{split}\label{eq:quenched_linear_response_7}
  \left((z - \mathbb{A}_\epsilon)^{-1} - (z - \mathbb{A}_0)^{-1}\right)v_0 = &\sum_{k=1}^{N-1} \epsilon^k  \sum_{m=1}^k \sum_{\substack{l_1 + \cdots + l_m = k \\ l_i \ge 1}} \left(\prod_{i=1}^m  (z - \mathbb{A}_0)^{-1}\mathbb{Q}_{l_i} \right)(z - \mathbb{A}_0)^{-1}v_0 \\
  &+ \left((z - \mathbb{A}_\epsilon)^{-1} - \mathbb{S}_\epsilon^{(N)} \right)v_0.
\end{split}\end{equation}
For each $k \in \{1, \dots, N-1\}$ we now define $v_0^{(k)}  \in \widetilde{\mathcal E}_0$ by
\[
  v_0^{(k)} = \intf_{D_\delta} \sum_{m=1}^k \sum_{\substack{l_1 + \cdots + l_m = k \\ l_i \ge 1}} \left(\prod_{i=1}^m   (z - \mathbb{A}_0)^{-1}\mathbb{Q}_{l_i} \right)(z - \mathbb{A}_0)^{-1}v_0 dz
\]
(recall that $v_0 \in \widetilde{\mathcal{E}}_N$). 
Furthermore, notice that $\xi(v_0^{(k)}) = 0$ almost surely for all $k$ as per the proof of Proposition \ref{prop:measurable_derivatives}.
By integrating \eqref{eq:quenched_linear_response_7} over $D_\delta$ and recalling \eqref{eq:quenched_linear_response_5} we get
\begin{equation}\label{eq:quenched_linear_response_8}
  v_\epsilon = v_0 + \sum_{k=1}^{N-1} \epsilon^{k} v_0^{(k)} + \intf_{D_\delta} \left((z - \mathbb{A}_\epsilon)^{-1} - \mathbb{S}_\epsilon^{(N)} \right)v_0 dz
\end{equation}
 in $\widetilde{\mathcal E}_0$.
Moreover, since 
 $D_\delta \subseteq \mathcal{V}_{\delta_0,a}(\mathbb{A}_0)$ it follows from \eqref{eq:quenched_linear_response_6}  that 
\begin{equation}\label{eq:quenched_linear_response_8b}
\begin{split}
  \norm{\intf_{D_\delta} \left((z - \mathbb{A}_\epsilon)^{-1} - \mathbb{S}_\epsilon^{(N)} \right)v_0 dz}_{\widetilde{\mathcal{E}}_0} &\le \sup_{z \in \mathcal{V}_{\delta_0,s}(\mathbb{A}_0)} \norm{(z- \mathbb{A}_\epsilon)^{-1} - \mathbb{S}_\epsilon^{(N)}(z)}_{L(\widetilde{\mathcal{E}}_N, \widetilde{\mathcal{E}}_0)} \norm{v_0}_{\widetilde{\mathcal{E}}_N} \\
  &\le C \norm{v_0}_{\widetilde{\mathcal{E}}_N} \abs{\epsilon}^{N-1 + \eta}.
\end{split}\end{equation}

Finally we show that $ v_0^{(k)}$   lies in $\mathcal E_0  \subset L^\infty (\Omega , E_0)$
 for each $k=1, \ldots ,N-1$.
Recall that $v_\epsilon \in \mathcal E_0 $ for every $\epsilon \in (-\epsilon _0,\epsilon _0)$. Thus $\epsilon ^{-1}(v_\epsilon - v_0)$ belongs to $ \mathcal E_0 $. 
Therefore, since  $ \mathcal E_0 $ isometrically injects into $\widetilde{\mathcal E}_0 $ (recall the argument above  \eqref{eq:gx1}), 
  it follows from the Taylor expansion  \eqref{eq:quenched_linear_response_8}, \eqref{eq:quenched_linear_response_8b}  that $\{\epsilon ^{-1}(v_\epsilon - v_0)\}_{\vert \epsilon \vert <\epsilon _0}$ is a Cauchy sequence in $ \mathcal E_0 $. 
Denote its limit by $v^\prime _0$, so that $\epsilon ^{-1}(v_\epsilon - v_0) - v_0' $ lies in $ \mathcal E_0 $ and  $\Vert \epsilon ^{-1}(v_\epsilon - v_0) - v_0' \Vert _{ \mathcal E_0 }\to 0$ as $\epsilon \to 0$.
Then by using again the fact that 
  $ \mathcal E_0 $ isometrically injects into $\widetilde{\mathcal E}_0$, we deduce that $v'_0$ equals the limit of $\{\epsilon ^{-1}(v_\epsilon - v_0)\}_{\vert \epsilon \vert <\epsilon _0}$ in $\widetilde{\mathcal E}_0$.
Hence, $v_0' = v_0^{(1)}$ by \eqref{eq:quenched_linear_response_8} and \eqref{eq:quenched_linear_response_8b}, which concludes that $v_0^{(1)}$ lies in $ \mathcal E_0 $. 
By induction (by considering $\epsilon ^{-k}( v_\epsilon - v_0 - \sum_{j=1}^{k-1} \epsilon^{j} v_0^{(j)} )$ instead of $\epsilon ^{-1}(v_\epsilon - v_0)$), we can show that $v_0^{(k)}$ also lies in $ \mathcal E_0 $ for each $k=2, \ldots , N-1$.
This completes the proof of  Theorem \ref{thm:quenched_linear_response} because \eqref{eq:quenched_linear_response_8} and \eqref{eq:quenched_linear_response_8b} hold with $\mathcal E_j$ in place of $\widetilde{\mathcal E}_j $.
\color{black}

\section{Applications to smooth random dynamical systems}\label{s:application}

In this section we shall apply Theorem \ref{thm:quenched_linear_response} to smooth random dynamical systems in order to obtain stability and differentiability results for their random equivariant probability measures. In particular, we will treat random Anosov maps and random U(1) extensions of expanding maps.
The treatments of these settings have much in common, and so we will discuss some general, abstract details in earlier subsections.

\subsection{Equivariant family of measures}
Let $M$ be a compact \color{black} connected \color{black} $\mathcal C^\infty$ Riemannian manifold and let $m$ denote the associated Riemannian probability measure on $M$. 
Fix a Lebesgue space $(\Omega, \mathcal{F}, \mathbb{P})$ and an  invertible, ergodic, $\mathbb P$-preserving map $\sigma : \Omega \to \Omega$.
For some $r \ge 1$ let $\mathcal{T} : \Omega \to \mathcal{C}^{r+1}(M,M)$ denote a $(\mathcal{F}, \mathcal{B}_{\mathcal{C}^{r+1}(M,M)})$-measurable map.
The RDS $(\mathcal T, \sigma )$ induced by $\mathcal T$ over $\sigma$ is explicitly written as  a measurable map
\[
\mathbb N_0 \times \Omega \times X \ni (n, \omega ,x)\mapsto \mathcal T^{(n)}_\omega (x), \quad \mathcal T^{(n)}_\omega :=\mathcal T_{\sigma^{n-1} \omega} \circ \cdots \circ \mathcal T _\omega,
\]
and since $\sigma$ is invertible,  the equivariance of a measurable family of probability measures $\{ \mu _\omega \}_{\omega \in \Omega}$ for $(\mathcal T,\sigma )$ is given as
\[
\mu_{\omega } \circ \mathcal T _{\omega }^{-1} = \mu_{\sigma \omega } \quad \text{ for almost every $\omega \in \Omega$},
\]
(recall Subsection \ref{ss:RDS}).

We aim to study the regularity of the dependence of $\{ \mu _\omega \}_{\omega \in \Omega}$ on the map $\mathcal{T}$ as $\mathcal{T}$ is fiber-wise varied in a uniformly $\mathcal{C}^N$ way for some $N \le r $.
To do this we shall realise equivariant families of probability measures as fixed points of (the lifts of) certain operator cocycles (linear RDSs) and then apply Theorem \ref{thm:quenched_linear_response}.
In particular, we shall consider the Perron-Frobenius operator cocycle associated to the RDS $( \mathcal{T} ,\sigma )$ on an appropriate Banach space.
Recall that the \emph{Perron-Frobenius operator} $\LL_T$ associated to a non-singular\footnote{Recall that a measurable map $T : M \to M$ is said to be non-singular (with respect to $m$) if $m(A) = 0$ implies that $m(T^{-1}(A)) = 0$.} measurable map  $T : M\to M$ is given by
\begin{equation}\label{eq:0914b}
\mathcal L_T f = \frac{\mathrm{d} [(f m) \circ T]}{\mathrm{d}m} \quad \text{for $f\in L^1(M,m)$},
\end{equation}
where $fm$ is a finite signed measure given by $(fm)(A) =\intf _A f dm$ for $A\in \mathcal A$ and $\mathrm{d}\mu /\mathrm{d}m$ is the Radon-Nikodym derivative of an absolutely continuous finite signed measure $\mu$.
Note that for each $M$-valued random variable $\psi$ whose distribution is $fm$ with some density function $f\in L^1(M,m)$, 
   $T(\psi )$ has the distribution  $(\mathcal L_T f) m$ (and thus,
    $\mathcal L_T$ is also called the \emph{transfer operator} associated with  $T$).
It is straightforward to see that
\begin{equation}\label{eq:pf_duality}
\intf_M \mathcal L_T fg dm = \intf_M fg\circ T dm \quad\quad \text{for $f\in L^1(M,m)$  and $g\in L^{\infty}(M,m)$},
\end{equation}
and that  $\mathcal L_T$ is an $m$-Markov operator,  where in an abuse of notation we are denoting the linear functional $f \in L^1(M,m) \mapsto \intf f dm$ by $m$.
In addition, $\LL_T$ is positive: if $f \in L^1(M,m)$ satisfies $f \ge 0$ almost  everywhere then $\LL_{T} f \ge 0$ almost  everywhere.

If\footnote{Notice that if $\det D_xT \ne 0$ then $T$ is automatically non-singular with respect to $m$.} $\det D_xT \ne 0$ for any $x\in M$ then since $T\in \mathcal{C}^{r+1}(M,M)$, one has $\LL_T \in L(\mathcal{C}^{r}(M))$ with
\begin{equation*}
  (\LL_T f)(x) =  \sum_{T(y) = x} \frac{f(y)}{\abs{\det D_y T}} \quad \text{for all $f \in \mathcal{C}^r(M)$}.
\end{equation*}
Hence from $\mathcal{T}$ we obtain a map $\LL_\mathcal{T}: \omega \mapsto \LL_{\mathcal{T}_\omega} : \Omega \to  L(\mathcal{C}^r(M))$, which is measurable by virtue of the following proposition. (We  postpone its proof until Appendix \ref{a:pt} because it is technical and  standard.)
  Let $\mathcal{N}^{r+1}(M,M)$ denote the set of $T \in \mathcal{C}^{r+1}(M,M)$ satisfying $\det D_xT \ne 0$ for any $x\in M$. 
\begin{proposition}\label{prop:cont_strong_op_topology}
The map $T \mapsto \LL_T$ is continuous on $\mathcal{N}^{r+1}(M,M)$ with respect to the strong operator topology on $L(\mathcal{C}^r(M))$.
\end{proposition}

Thus if we demand that $\mathcal{T} \in \mathcal{N}^{r+1}(M,M)$ almost surely then $( \LL_{\mathcal{T}}, \sigma )$ is an $m$-Markov operator cocycle on $\mathcal{C}^r(M)$, which we shall call the \emph{Perron-Frobenius operator cocycle} (on $\mathcal{C}^r(M)$) associated to $\mathcal{T}$.
In order to apply the theory of Section \ref{sec:spectral_approach_stability} we require that the Perron-Frobenius operator cocycle is bounded and $m$-mixing. This later condition will entail some mixing hypotheses on our random systems.
However, as in the deterministic case, in order to realise the mixing of the RDS in operator theoretic terms we may be forced to consider the Perron-Frobenius operator cocycle on an alternative Banach space. Specifically, we shall seek Banach spaces $(X, \norm{\cdot}_X)$ satisfying the following conditions:
\begin{enumerate}[label=\textnormal{(s\arabic*)}]
  \item \label{en:S1} $\mathcal{C}^r(M)$ is dense in $X$ with $\mathcal{C}^r(M) \into X$;
  \item \label{en:S2} The embedding $\mathcal{C}^r(M) \into (\mathcal{C}^\infty(M))^*$ given by the map $h \in \mathcal{C}^r(M) \mapsto (g \in \mathcal{C}^\infty(M) \mapsto \intf gh dm)$ continuously extends to an embedding $X \into (\mathcal{C}^\infty(M))^*$.
\end{enumerate}
It is clear that any $X$ satisfying \ref{en:S1} must be separable.
Moreover, we note that the functional $\varphi \in (\mathcal{C}^\infty(M))^* \mapsto \varphi(1_M)$ is continuous on $(\mathcal{C}^\infty(M))^*$, and yields $m$ when pulled-back via the embedding $\mathcal{C}^r(M) \into (\mathcal{C}^\infty(M))^*$ that is described in \ref{en:S2}.
Hence, if \ref{en:S2} holds, and so we have an embedding $X \into (\mathcal{C}^\infty(M))^*$ that continuously extends the $\mathcal{C}^r(M) \into (\mathcal{C}^\infty(M))^*$, then $m$ induces a continuous linear functional on $X$. In particular, we may speak of $m$-Markov operators in $L(X)$.
The following proposition gives a sufficient condition for an $m$-Markov operator  in $L(\mathcal C^r(M))$ to be extended to an $m$-Markov operator in $L(X)$.

\begin{proposition}\label{prop:measurable_extension}
Let $(A, \sigma )$ be a bounded, $m$-Markov operator cocycle on $\mathcal C^r(M)$ and $X$  a Banach space satisfying \ref{en:S1} and \ref{en:S2}.  
Suppose that
  \begin{equation}\label{eq:s3}
    \esssup_\omega \sup_{\substack{f \in \mathcal{C}^r(M) \\ \norm{f}_X = 1}} \norm{A (\omega )f}_{X} < \infty.
  \end{equation}
  Then $A$ almost surely extends to a unique, bounded operator on $X$ such that $\omega \mapsto A(\omega ) : \Omega \to  L(X)$ is strongly measurable. Consequently, $( A,\sigma )$ is a bounded, $m$-Markov operator cocycle on $X$ such that 
    \begin{equation}\label{eq:s3b}
      \esssup_\omega \norm{A (\omega )}_{L(X)} <\infty.
      \end{equation}
  \end{proposition}
  \begin{proof}
  It is clear that $A$ almost surely extends to a unique, bounded operator on $X$, and that
    \begin{equation*}
      \esssup_\omega \norm{A (\omega )}_{L(X)} = \esssup_\omega \sup_{\substack{f \in \mathcal{C}^r(M) \\ \norm{f}_X = 1}} \norm{A (\omega )f}_{X} < \infty.
    \end{equation*}
    That $A$ is almost surely  $m$-Markov in $L(X)$ follows straightforwardly from the fact that $A$ is almost surely  $m$-Markov in $L(\mathcal{C}^r(M))$, and that $m$ uniquely extends to a continuous linear functional on $X$.
    Hence, it only remains to show that $\omega \mapsto A (\omega )$ is strongly measurable in $L(X)$.
    Suppose that $f \in X$. Then there exists a sequence $\{f_n\}_{n \in \N} \subset \mathcal{C}^r(M)$ with limit $f$ in $X$.
    For each $n$ the map $\omega \mapsto A (\omega ) f_n$ is $(\mathcal{F}, \mathcal{B}_{\mathcal{C}^r(M)})$-measurable, and so it must be $(\mathcal{F}, \mathcal{B}_{X})$ measurable too due to \ref{en:S1}.
    Moreover, for almost every $\omega$ we have
    \begin{equation*}
      \lim_{n \to \infty} \norm{A (\omega ) f_n - A (\omega ) f}_X = 0,
    \end{equation*}
    which is to say that $\omega \mapsto A (\omega ) f$ is the almost everywhere pointwise limit (in $X$) of $(\mathcal{F}, \mathcal{B}_{X})$-measurable functions.
    Hence $\omega \mapsto A (\omega ) f$ is $(\mathcal{F}, \mathcal{B}_{X})$-measurable since $(\Omega , \mathcal F, \mathbb P)$ is a Lebesgue space (in particular complete).
    That $\omega \mapsto A (\omega )$ is strongly measurable in $L(X)$ then follows from Lemma \ref{lemma:strong_measurability_summary} and the fact that $f \in X$ was arbitrary.
  \end{proof}

Hence, 
by Propositions \ref{prop:cont_strong_op_topology} and \ref{prop:measurable_extension}, 
if $\mathcal T: \Omega \to \mathcal N^{r+1}(M,M)$ is measurable and $X$ satisfies \ref{en:S1} and \ref{en:S2}, then
the Perron-Frobenius operator cocycle $( \LL_{\mathcal{T}} ,\sigma )$ on $\mathcal C^r(M)$ can be extended to a bounded, $m$-Markov operator cocycle on $X$.
Compare also \eqref{eq:s3b} 
 with (QR1).

The following proposition will help us to describe the relationship between the equivariant family of probability measures for $(\mathcal{T} , \sigma )$ and the fixed point of the lift of a bounded, $m$-mixing Perron-Frobenius operator cocycle $(\mathcal L_{\mathcal{T} }, \sigma )$.
\begin{definition}\label{dfn:1208}
Assume that $X$ satisfies \ref{en:S1}.
$A\in L(X)$ is called \emph{positive} if $A(X_+) \subset X_+$, where 
$X_+$ is the completion of $\{ f \in \mathcal{C}^r(M) : f \ge 0\}$ in $\norm{\cdot}_X$.
 An operator cocycle $(A,\sigma)$ is called \emph{positive} if $A$ is almost surely positive.
 Furthermore, a distribution $f\in (\mathcal C^\infty(M))^*$ is called \emph{positive} if $ f( g) \geq 0$ for every $g\in \mathcal C^\infty(M)$ such that $g\geq 0$.
\end{definition}

\begin{proposition}\label{prop:existence_equivariant_measure}
  Let $X$ be a Banach space satisfying \ref{en:S1} and \ref{en:S2} and $(A,\sigma)$  a bounded, $m$-Markov operator cocycle on $X$.
  Suppose that $(A,\sigma)$ is positive and $m$-mixing  and that $h$ is the unique $m$-normalised fixed point of the lift $\mathbb{A} : \mathcal X\to \mathcal X$ on $\mathcal X\subset L^\infty (\Omega ,X)$ (recall \eqref{eq:1122c} for its definition).
  Then there exists a 
   measurable family of Radon probability measures $\{ \mu _\omega \} _{\omega \in \Omega}$  such that $h(\omega )(g) = \int g d\mu _\omega$ for every $g\in \mathcal C^\infty (M)$ and
    almost every $\omega$.
  \end{proposition}
  \begin{proof}
    Notice that the set
    \begin{equation*}
      \mathcal{D} = \{ f \in L^\infty(\Omega,X) : m(f) = 1 \text{ and } f \in X_+ \text{ almost surely}\}
    \end{equation*}
    is almost surely invariant under $A(\omega)$ since $(A,\sigma)$ is bounded, positive, and $m$-Markov.
    Hence we may carry out the construction of $h$ in Proposition \ref{prop:spectral_properties_lift} with $\mathcal{D}$ in place of $\mathcal{X}_1$, 
     to conclude that $h \in X_+$ almost surely.
    Thus, 
     there exists $\{ f_k\}_{k \in \N} \subseteq L^\infty (\Omega ,\mathcal{C}^r(M))$ such that $f_k(\omega ) \ge 0$ and $\intf f_k (\omega ) dm = 1$ for every $k$ and so that $\lim_{k \to \infty} f_k (\omega )= h(\omega )$ in $X$ for almost every $\omega$.
    As $X \into (\mathcal{C}^\infty(M))^*$ it follows that $\lim_{k \to \infty} f_k (\omega )= h(\omega )$ in the sense of distributions as well. 
    Thus for any positive $g \in \mathcal{C}^\infty(M)$  
      we have
    \begin{equation}\label{eq:existence_equivariant_measure_1}
   h(\omega )(g)= \lim_{k \to \infty} f_k(\omega ) (g) = \lim_{k \to \infty} \intf f_k(\omega ) \cdot g dm
    \end{equation}
    (recall the embedding of $\mathcal C^r(M)$ in \ref{en:S2}).
    As $f_k(\omega )$ and $g$ are positive, it follows from \eqref{eq:existence_equivariant_measure_1} that $h(\omega )(g)  \ge 0$ for every such $g$.
    Hence $h(\omega )$ is a positive distribution for almost every $\omega$. 
    On the other hand, as is well known, 
 for any positive $f\in (\mathcal C^\infty (M))^*$, one can find a positive Radon measure $\mu _f$ such that $f(g) =\int gd\mu _f$ for every $g\in \mathcal C^\infty (M)$. 
 We denote by $\mu _\omega$ the 
  positive Radon measure corresponding to $h(\omega )$.

 To see that $\mu _\omega $ is a probability measure for almost every $\omega$,
     we note that by
      \eqref{eq:existence_equivariant_measure_1} and as $\intf f_k(\omega ) dm = 1$ for every $k$
       we have
    \begin{equation*}
      \mu_\omega(M) = h(\omega )(1_M) = \lim_{k \to \infty} \intf f_k(\omega) dm = 1.
    \end{equation*}
Finally, $\{ \mu _\omega \} _{\omega \in \Omega}$  is a measurable family on the complete probability space $(\Omega , \mathcal F, \mathbb P)$ 
 because for any $A\in \mathcal B_M$, by using
      \eqref{eq:existence_equivariant_measure_1} again  we have
\[
\mu _\omega (A) = h(\omega )(1_A) 
 = \lim_{k \to \infty} \intf _A f_k(\omega )  dm
\]
for almost every $\omega$, 
while for every $k$, $\omega \mapsto f_k(\omega ) :\Omega \to  \mathcal C^r(M)$ is 
measurable   and $f\mapsto  \intf _A f  dm : \mathcal C^r(M) \mapsto \mathbb C$ is continuous, so that $\omega \mapsto \intf _A f_k(\omega )  dm$ is measurable too.
  \end{proof}

Hence if $X$ satisfies \ref{en:S1} and \ref{en:S2}, and  the Perron-Frobenius operator cocycle $( \LL_{\mathcal{T}} , \sigma )$ on $X$ is $m$-mixing, then we obtain a
  measurable family of Radon probability measures $\{ \mu _\omega \} _{\omega \in \Omega}$ such that $h: \omega \mapsto (g \mapsto \int g d\mu _\omega )$ is in $L^\infty (\Omega , X)$.
  Furthermore, $\{ \mu _\omega \} _{\omega \in \Omega}$ is equivariant because
   it follows from  
      \eqref{eq:existence_equivariant_measure_1} that for any $A\in \mathcal B_M$ and almost every $\omega$
\[
\mu _\omega (T_\omega ^{-1} (A)) =h(\omega ) \big(1_{T_\omega ^{-1} (A)}\big)=
\lim _{k\to\infty}\int f_k(\omega ) \cdot 1_{ A} \circ T_\omega dm.
\]
Due to \eqref{eq:pf_duality}, the continuity of $\mathcal L_{T_\omega} : X\to X$ and the fact that $h$ is the fixed point of the lift of $(\mathcal L_{\mathcal T}, \sigma )$, this coincides with
\[
\lim _{k\to\infty}\int \mathcal L_{T_\omega} f_k(\omega ) \cdot 1_{ A} dm
=\lim _{k\to\infty} \mathcal L_{T_\omega} f_k(\omega ) (1_{ A})
= \mathcal  L_{T_\omega }h(\omega ) (1_A) =h(\sigma \omega )(1_A) =\mu _{\sigma \omega } (A).
\]

\subsection{The conditions (QR4) and (QR5)} 

In this subsection we discuss a sufficient condition for a family of Perron-Frobenius operator cocycles $\{ (\mathcal L_{\mathcal T_\epsilon } ,\sigma )\} _{\epsilon \in [-1,1]}$ to satisfy (QR4) and (QR5).
We emphasise that these conditions hold rather independently of how the underlying random dynamics $(\mathcal T_\epsilon, \sigma )$ behaves (see Proposition \ref{prop:1209a} for precise statement), so we treat (QR4) and (QR5) here as a final preparation before specialising to our applications.
For simplicity, throughout  this subsection we assume that $M$ is a $d$-dimensional torus $\mathbb T^d$.
One may straightforwardly remove this assumption by considering a partition of unity,  refer to e.g.~\cite{gouezel2006banach, baladi2018dynamical} (see also Appendix \ref{a:pt}).

Notice that (QR4) and (QR5) are conditions for a single iteration $\mathcal L_{T_{\epsilon ,\omega }}$ (not for $\mathcal L_{T_{\epsilon , \sigma ^{n-1} \omega}} \circ \cdots \circ \mathcal L_{T_{\epsilon , \omega}}$, $n\in \mathbb N$), 
 and so clear observations may be found in the non-random setting.
Fix $r\geq 1$ and $1\leq s \leq r$, and
we will consider $T\in \mathcal C^N([-1,1], \mathcal C^{r+1}(\mathbb T^d ,\mathbb T^d))$.
Let $1\leq N\leq s$ be an integer, and $E_j$, $j \in \{0, \dots, N\}$,  Banach spaces with $E_j \into E_{j-1}$ for each $j \in \{1, \dots, N\}$ satisfying  the following conditions:
\begin{enumerate}[label=\textnormal{(S\arabic*)}]
  \item  \ref{en:S1} holds with $E_j$ in place of $X$ for each $j \in \{0, \dots, N\}$;
  \item  \ref{en:S2} holds with $E_j$ in place of $X$ for each $j \in \{0, \dots, N\}$;
  \item
  \label{en:S3}  There are constants $C>0$ and $0\leq \rho \leq r-N$ 
   such that 
\[
\Vert uf \Vert _{E_j} \leq C\Vert u \Vert _{\mathcal C^{\rho  +j}} \Vert f\Vert _{E_j} \quad \text{for each  $u, f\in \mathcal C^r(\mathbb T^d)$ and $j\in \{0,\ldots ,N\}$}.
\]
\item \label{en:S4} There is a constant $C>0$ such that 
\[
\left\Vert \frac{\partial}{\partial x_l } f\right\Vert _{E_{j-1}} \leq C\Vert f\Vert _{E_j} \quad \text{for each $ f\in \mathcal C^r(\mathbb T^d)$, $ l\in \{1,\ldots , d\}$ and $ j\in \{ 1, \ldots , N\}$}.
\]
\end{enumerate}
Observe that all conditions   $\mathrm{(S1)}$-$\mathrm{(S4)}$ are not for the operators $\mathcal L_{T_\epsilon }$, $\epsilon \in [-1,1]$, with $T_\epsilon := T(\epsilon )$, but for the spaces $E_j$, $j\in \{0,\ldots ,N\}$, so the following proposition is quite useful 
 in our applications.
Note that  if
  \begin{equation}\label{eq:1209b}
  \norm{\mathcal L_{T_\epsilon }f}_{E_j} < C_\epsilon \Vert f\Vert _{E_j}\quad \text{   for each $f \in \mathcal{C}^r(\mathbb T^d)$, $j \in \{0, \dots, N\}$ and $\vert \epsilon \vert \leq 1$},
  \end{equation}
then it follows from Proposition \ref{prop:measurable_extension} that $\mathcal L_{T_\epsilon}$ is a bounded operator on $E_j$ for each  $j \in \{0, \dots, N\}$ and $\vert \epsilon \vert \leq 1$.

\begin{proposition}\label{prop:1209a}
Let  $N$ be a positive integer, $T \in \mathcal C^N([-1,1], \mathcal C^{r+1}(\mathbb T^d ,\mathbb T^d))$, and  $E_j$, $j \in \{0, \dots, N\}$,  Banach spaces with $E_j \into E_{j-1}$ for each $j \in \{1, \dots, N\}$ satisfying $\mathrm{(S1)}$-$\mathrm{(S4)}$.
Suppose that  $T_\epsilon \in \mathcal N^{r+1}(\mathbb T^d ,\mathbb T^d)$ for each $\epsilon \in [-1,1]$ and \eqref{eq:1209b} holds.
Then 
$\epsilon \mapsto \mathcal L_{T_\epsilon } f$ is in
$\mathcal C^j([-1,1], E_{i-j})$
 for each $j \in \{1, \dots, N\}$, $i\in \{ j,\ldots , N\}$ and $f\in E_i$.
\end{proposition}

Before starting the proof of Proposition \ref{prop:1209a}, we discuss a consequence of Proposition \ref{prop:1209a} with respect to the conditions (QR4) and (QR5).
Let  
$\{(\mathcal T_\epsilon , \sigma )\}_{\epsilon \in [-1,1]}$ be a family of RDSs such that
for almost every $\omega$, 
 the map
  $\epsilon \mapsto T_{\epsilon , \omega }:= \mathcal T_\epsilon (\omega )$ is in $\mathcal C^N([-1,1], \mathcal C^{r+1}(\mathbb T^d ,\mathbb T^d))$.
 Let $E_j$, $j \in \{0, \dots, N\}$ be  Banach spaces with $E_j \into E_{j-1}$ for each $j \in \{1, \dots, N\}$ satisfying $\mathrm{(S1)}$-$\mathrm{(S4)}$.
 We suppose that
 \begin{enumerate}
  \item[(S5)] 
  \label{S:5} 
$T_{\epsilon ,\omega } \in \mathcal N^{r+1}(\mathbb T^d ,\mathbb T^d)$ for each $\epsilon \in [-1,1]$ and almost every $\omega$.
Furthermore, 
\eqref{eq:s3} holds with $E_j$ and $\mathcal L_{T_{\epsilon , \omega}}$ in place of $X$ and $A(\omega )$ for every $j\in \{ 0,\ldots , N\}$ and $\vert \epsilon \vert \leq 1$. 
\end{enumerate}
Then, it follows from Proposition \ref{prop:measurable_extension}, the Perron-Frobenius operator cocyles $(\mathcal L_{\mathcal T_{\epsilon }}, \sigma )$, $\epsilon \in [-1,1]$, can be extended to a bounded operator cocycles on each $E_j$, and (QR1) holds for these operator cocycles by virtue of \eqref{eq:s3b}.

For each $j\in \{ 0, \ldots , N\}$, $i\in \{ j,\ldots , N\}$ and almost every $\omega$, it follows from Proposition \ref{prop:1209a} that we can define $Q_j(\omega ) : E_i \to E_{i-j}$ by
\[
Q_j (\omega ) f = \frac{1}{j!} \left(\frac{d^j}{d\epsilon ^j} \mathcal L_{T_{\epsilon ,\omega } } f \right) _{\epsilon =0} \quad \text{for $f\in E_i$}.
\]
By the definition,
  it is straightforward to see that 
   for all $\epsilon \in [-1,1]$ and $2\leq j\leq N$,
  \begin{equation*}
    \esssup_\omega \norm{\mathcal L_{T_{\epsilon ,\omega}} - \mathcal L_{T_{0 ,\omega}}  }_{L(E_N, E_{N-1})} \le C \abs{\epsilon},
  \end{equation*}
  and
  \begin{equation*}
    \esssup_{\omega} \norm{\mathcal L_{T_{\epsilon ,\omega}} - \mathcal L_{T_{0 ,\omega}} - \sum_{k=1}^{j-1} \epsilon^k Q_k(\omega) }_{L(E_N, E_{N-j})} \le C \abs{\epsilon}^j.
  \end{equation*}
  To summarise the above argument, we conclude the following:
  \begin{corollary}\label{cor:1211}
Suppose that 
 $\mathrm{(S1)}$-$\mathrm{(S5)}$ hold for  the family of Perron-Frobenius operator cocyles $\{ (\mathcal L_{\mathcal T_{\epsilon }}, \sigma )\}_{\epsilon \in [-1,1]}$ on Banach spaces $E_j$, $j\in \{0, \ldots ,N\}$.
 Then
  $\mathrm{(QR1)}$, $\mathrm{(QR4)}$ and $\mathrm{(QR5)}$  hold.
  \end{corollary}

We now return to the proof of Proposition \ref{prop:1209a}.
\begin{proof}[Proof of Proposition \ref{prop:1209a}]
Fix $1\leq \sigma \leq N$, $1\leq \rho \leq r$,  $f\in \mathcal C^{\sigma}([-1,1] , \mathcal C^\rho (\mathbb T^d))$, $g\in \mathcal C^\rho (\mathbb T^d)$ and $1\leq l\leq d$
for the time being 
(notice that  this $f$ is different from $f$ in the statement of Proposition \ref{prop:1209a} in the sense that this $f$ depends on $\epsilon \in [-1,1]$).
 We  simply denote $\frac{d^a}{d\epsilon ^a} f \in \mathcal C^{\sigma -a}([-1,1] , \mathcal C^\rho (\mathbb T^d))$ by $f^{(a)}$ for each  integer $a\in [0, \sigma ]$.
 We also simply denote by $\partial _l g$ the partial derivative of  $g$ with respect to the $l$-th coordinate and 
let $\partial ^\alpha = \partial ^{\alpha _1}_1 \cdots \partial ^{\alpha _d}_d$ and $\vert \alpha \vert = \alpha _1 +\cdots + \alpha _d$ for each multi-index $\alpha  =(\alpha _1 ,\ldots , \alpha _d) \in \mathbb N_0^d$.
Then for each $ \epsilon \in [-1,1]$ and $x\in \mathbb T^d$,
\[
\partial _{l} (f ^{(1)}_\epsilon )(x)= \partial _\epsilon \partial _{l} \tilde f  (\epsilon , x) =  \partial _{l}\partial _\epsilon \tilde f  (\epsilon , x)
\]
where $\tilde f: [-1,1] \times \mathbb T^d \to \mathbb T^d$ is given by $\tilde f (\epsilon , x) = f_\epsilon (x)$. 
  (Since $f^{(1)} \in 
    \mathcal C^{0}([-1,1] , \mathcal C^1 (\mathbb T^d))$,
  it is straightforward to see that  the first equality holds and
    $(\epsilon ,x) \mapsto \partial _{l} (f ^{(1)}_\epsilon )(x) $ is continuous.
 The second equality also immediately follows from these observations together with   Schwarz-Clairaut's theorem on equality of mixed partials.)
 In particular, 
 \begin{equation}\label{eq:1211} 
 f ^{(1)} =(\epsilon \mapsto \partial _\epsilon \tilde f  (\epsilon , \cdot ) )\quad \text{in $\mathcal C^{\sigma -1 }([-1,1], \mathcal C^{\rho }(\mathbb T^d))$}.
 \end{equation}
Furthermore, it is also straightforward to see that
 the map $\epsilon \mapsto \partial _l f_ \epsilon $ is in $ \mathcal C^\sigma ([-1,1],  \mathcal C^{\rho -1}(\mathbb T^d))$, which  we denote  by $\partial _l f$ as a slight abuse of notation, and that
\begin{equation}\label{eq:1013b2} 
(\partial _{l}f )^{(1)}= \partial _{l} (f^{(1)})  \quad \text{in $\mathcal C^{\sigma -1}([-1,1] , \mathcal C^{\rho -1}(\mathbb T^d))$},
\end{equation}
\begin{equation}\label{eq:1013b} 
(\det  DT ) ^{(1)} = \det DT ^{(1)} \quad \text{in $\mathcal C^{N-1}([-1,1] , \mathcal C^{r-1}(\mathbb T^d))$}.
\end{equation}
Moreover, we denote by  $T_{(l)}\in \mathcal C^N([-1,1], \mathcal C^{r+1}(\mathbb T^d))$  the map $\epsilon \mapsto (x\mapsto T_{(l), \epsilon } (x))$, where $ T_{(l), \epsilon } (x)$ is the $l$-th coordinate of $T_\epsilon (x) \in \mathbb T^d$ (under the identification of $\mathbb T^d$ with $\mathbb R^d$).
Finally we define a map $\mathbf Lf: [-1,1] \to \mathcal C^{\rho }(\mathbb T^d)$
 by
\[
(\mathbf Lf )_ \epsilon  = \mathcal L_{T_ \epsilon } f_\epsilon \quad \text{for  $\epsilon \in [-1,1]$},
\]
which is well-defined by virtue of \eqref{eq:1209b}.
The following is the key lemma for the proof of Proposition  \ref{prop:1209a}.
\begin{lemma}\label{prop:2.1}
For each $f\in \mathcal C^\sigma ([-1,1], \mathcal C^\rho (\mathbb T^d))$ with  $\sigma \in [1,N]$ and $\rho \in [1,r]$,  $(\mathbf L f) ^{(1)}$ exists in $\mathcal C^{\sigma -1} ([-1,1], \mathcal C^{\rho -1}(\mathbb T^d))$, and  is of the form
\begin{equation}\label{eq:1010}
(\mathbf L f) ^{(1)}= \mathbf L \left(\sum _{\vert \alpha \vert \leq 1} J_{0,\alpha } \cdot \partial ^\alpha f  +  \sum _{\vert \alpha \vert \leq 1} J_{1, \alpha } \cdot \partial ^\alpha  f^{(1)}  \right), 
\end{equation}
where  $J_{k ,\alpha  } $ is in $\mathcal C^{N-1}([-1,1] ,\mathcal C^{r-1} (\mathbb T^d))$ 
being a polynomial function of  $\partial ^\beta T_{(l )}$, $\partial ^\beta   T_{(l )}^{(1)}$ ($1\leq l \leq d$, $\vert \beta \vert \leq 2$) and $( \det DT)^{-1}$
 for each $k=0,1$  and multi-index $\alpha $ with $\vert \alpha\vert \leq 1$.
\end{lemma}
\begin{proof}
Observe that 
$
\det D_x T_\epsilon  >0$ for all $\vert \epsilon \vert \leq 1$ and $x\in \mathbb T^d$ or
$\det D_x T_\epsilon  <0$ for all $\vert \epsilon \vert \leq 1$ and $x\in \mathbb T^d$
because $T_\epsilon \in \mathcal N^{r+1}(\mathbb T^d,\mathbb T^d)$ for each $\vert \epsilon \vert \leq 1$ and $\epsilon \mapsto T_\epsilon$ is continuous.
We only consider the former case because the other case is similar.
Also, we only show \eqref{eq:1010} around $\epsilon =0$ to keep our notation simple (the general case can be literally treated).
We first note that there is $\epsilon _0 >0$, $B\in \mathbb N$, a finite covering $\{ U_\lambda \} _{\lambda \in \Lambda }$ of $\mathbb T^d$ and $\mathcal C^{r+1}$ maps $(T _\epsilon )^{-1}_{\lambda ,b} : U_\lambda \to (T _\epsilon )^{-1}_{\lambda ,b} (U_\lambda ) $ for $\vert \epsilon \vert <\epsilon _0$, $\lambda \in \Lambda$ and $b\in \{ 1,\ldots , B\}$ such that 
 for each $\vert \epsilon \vert <\epsilon _0$,  $\lambda \in \Lambda$, $b\in \{ 1,\ldots , B\}$  and $g\in \mathcal C^r(\mathbb T^d)$,
\begin{equation}\label{eq:1210a}
T _\epsilon \circ (T _\epsilon )^{-1}_{\lambda ,b} (x) = x\quad \text{on $ U_\lambda$}
\end{equation}
and 
\begin{equation}\label{eq:1210b}
\mathcal M_1(\epsilon , g)(x) := \sum _{T_\epsilon (y) =x} g(y) = \sum _{b=1}^B g\circ (T _\epsilon )^{-1}_{\lambda ,b} (x) \quad \text{on $ U_\lambda$}
\end{equation}
(because $T_0 \in \mathcal N^{r+1}(\mathbb T^d,\mathbb T^d)$ and $\mathbb T^d$ is compact, see Appendix \ref{a:pt} for detail).
Note also that if we define $\mathcal M_2: [-1,1]\times  \mathcal C^r(\mathbb T^d) \to \mathcal C^r(\mathbb T^d)$ by
\[
\mathcal M_2 (\epsilon , g) = \frac{g}{ \det DT_\epsilon }
\quad \text{for $\epsilon \in [-1,1]$, $g\in \mathcal C^r(\mathbb T^d)$,}
\]
then for each $f\in \mathcal C^\sigma ([-1,1],\mathcal C^r(\mathbb T^d))$ and $\vert \epsilon \vert <\epsilon _0$,
we have
\[
(\mathbf L f)_\epsilon   =\mathcal M_1(\epsilon , \mathcal M_2(\epsilon , f_\epsilon )). 
\]
Notice that both $\mathcal M_1$ and $\mathcal M_2$ are linear with respect to $g\in \mathcal C^r(\mathbb T^d)$.
Hence 
it follows from the chain rule for (Fr\'echet) derivatives that
\begin{equation}\label{eq:1211d}
(\mathbf L f)_\epsilon ^{(1)}  = \partial _\epsilon \mathcal M_1(\epsilon , \mathcal M_2(\epsilon , f_\epsilon )) +  \mathcal M_1\left(\epsilon , \partial _\epsilon \mathcal M_2(\epsilon , f_\epsilon ) + \mathcal M_2\left(\epsilon ,  f_\epsilon ^{(1)} \right)\right) 
\end{equation}
if the derivatives exists, where $\partial _\epsilon =\frac{\partial}{\partial \epsilon }$

Now we calculate $ \partial _\epsilon \mathcal M_1$ and $ \partial _\epsilon \mathcal M_2$. 
We first show that
\begin{equation}\label{eq:1210a2}
\partial _\epsilon \mathcal M_1(\epsilon , g)  =\mathcal M_1\left( \epsilon , -\sum _{l =1}^d \partial _{l} g \cdot \frac{\sum _{k=1}^d (\mathrm{adj} ( DT_\epsilon)) _{l , k} \cdot T_{(k),\epsilon } ^{(1)}  }{\det D T_\epsilon } \right), 
\end{equation}
where $\mathrm{adj} (A)$ is the adjugate matrix (i.e.~the transpose of the cofactor matrix) of a matrix $A$.
 By \eqref{eq:1211} (with $\epsilon \mapsto \mathcal M_1 (\epsilon , g)$ in place of $f$), \eqref{eq:1210b} and  the chain rule for  derivatives, 
  we have
\begin{equation}\label{eq:1210a3}
\partial _\epsilon  \mathcal M_1 (\epsilon , g)(x)
=  \partial _\epsilon \widetilde{  \mathcal M}_{1,g} (\epsilon , x)
= \sum _{b=1}^B \sum _{l =1}^d \partial _{l} g \circ  (T _\epsilon )^{-1} _{b,\lambda}(x) \cdot 
\partial _\epsilon ((T _\epsilon )^{-1}_{b,\lambda})_{(l)} (x), \quad x\in U_\lambda ,
\end{equation}
 where $((T _\epsilon )^{-1}_{b,\lambda} )_{(l)}(x)$ is the $l$-th coordinate of $(T _\epsilon )^{-1}_{b,\lambda} (x)$ and $\widetilde{  \mathcal M}_{1,g} (\epsilon , x) =  \mathcal M _1(\epsilon , g)(x)$.
On the other hand, by differentiating the $l$-th coordinate of \eqref{eq:1210a} for $1\leq l \leq d$ we have
\[
T_{(\ell ) ,\epsilon }^{(1)} (y)  + \sum _{k =1}^d \partial _{k } T_{(\ell ),\epsilon  }(y)  \cdot \partial _\epsilon ((T _\epsilon )^{-1}_{b,\lambda} )_{(k)} (x)=0, \quad y=(T _\epsilon )^{-1}_{b,\lambda} (x), \; x\in U_\lambda .
\]
In the matrix form (under the identification of $\mathbb T^d$  with $\mathbb R^d$), this can be written as 
\[
  T^{(1)}_\epsilon (y)  + D_y T_{\epsilon  }   \left[ \partial _\epsilon (T _\epsilon )^{-1}_{b,\lambda}(x) \right] =0, \quad y =(T _\epsilon )^{-1}_{b,\lambda} (x), \; x\in U_\lambda,
\]
where we see $ T^{(1)}_\epsilon (y)$ and $ \partial _\epsilon (T _\epsilon )^{-1}_{b,\lambda} (x)$ as column vectors.
Thus, since $ A^{-1} = (\det A)^{-1} \mathrm{adj}(A)$  for any invertible matrix $A$, we have
\begin{equation}\label{eq:1210a4}
 \partial _\epsilon (T _\epsilon )^{-1}_{b,\lambda}(x)  
 =- (\det D_yT_\epsilon )^{-1} \mathrm{adj}(D_yT_\epsilon )    \left[  T^{(1)}_\epsilon (y) \right], \quad y =(T _\epsilon )^{-1}_{b,\lambda} (x), \; x\in U_\lambda.
\end{equation}
\eqref{eq:1210a2} immediately follows from \eqref{eq:1210a3} and  \eqref{eq:1210a4}.
Furthermore,
by the quotient rule for derivatives and \eqref{eq:1013b} we have
\begin{equation}\label{eq:1211b}
\partial _\epsilon \mathcal M_2 (\epsilon , g)  =-\frac{g  \cdot \det D T^{(1)}_{\epsilon } }{(\det D T_{\epsilon } )^2}
\end{equation}
and
\begin{equation}\label{eq:1211c}
\partial _l \mathcal M_2 (\epsilon , g) = \frac{\partial _l g}{ \det DT_\epsilon } - \frac{g\cdot \partial _l(\det DT_\epsilon)}{ (\det DT_\epsilon)^2 }.
\end{equation}
The conclusion immediately follows from \eqref{eq:1211d}, \eqref{eq:1210a2}, \eqref{eq:1211b}  and \eqref{eq:1211c}.
\end{proof}

Now we complete the proof of Proposition \ref{prop:1209a}.
We first consider the case when $f\in \mathcal C^r(\mathbb T^d)$.
We will show by induction that for each $1\leq k \leq j$,  $(\mathbf L f)^{(k)}$ exists and is of the form
\begin{equation}\label{eq:1014d2}
(\mathbf L f)^{(k)} = \mathbf L \left(\sum _{\vert \alpha \vert \leq k} \widehat J_{k,\alpha } \cdot \partial ^\alpha f    \right), 
\end{equation}
where $\widehat J_{k,\alpha }$ is  in $\mathcal C^{N-k}([-1,1],\mathcal C^{r-k} (\mathbb T^d))$ 
being a polynomial function of  $\partial ^\beta   T_{(l )}^{(k')}$ ($1\leq l \leq d$, $0\leq k'\leq k$, $\vert \beta \vert \leq k+1$) and $( \det DT)^{-1}$
 for each multi-index $\alpha $ with $\vert \alpha\vert \leq k$.
 \eqref{eq:1014d2} for $k=1$ is an immediate consequence of Proposition \ref{prop:2.1} (notice that $f$ in Proposition \ref{prop:2.1} depended on $\epsilon$ while $f$ here does not).
Suppose that $k\geq 2$ and \eqref{eq:1014d2} holds with $k-1$ instead of $k$.
 Then, by Proposition \ref{prop:2.1}, we have
\begin{align*}
(\mathbf L f)^{(k )} = 
\mathbf L \left(\sum _{\vert \alpha \vert \leq 1} J_{ 0,\alpha } \cdot \partial ^\alpha \left(\sum _{\vert \beta \vert \leq k-1 } \widehat J_{k-1 ,\beta } \cdot \partial ^\beta f \right)  
+  \sum _{\vert \alpha \vert \leq 1} J_{1, \alpha } \cdot \partial ^\alpha  \left(\sum _{\vert \beta \vert \leq k -1} \widehat J_{k-1 ,\beta }^{(1)} \cdot \partial ^\beta f    \right)  \right).
\end{align*}
Therefore,  \eqref{eq:1014d2} immediately follows from \eqref{eq:1013b2} and \eqref{eq:1013b}.
Furthermore, $\epsilon \mapsto \frac{d^j}{d\epsilon ^j} \mathcal L_{T_\epsilon } f = (\mathbf L f) ^{(j)}_\epsilon $ exists as an element of $\mathcal C^{0}([-1,1], E_{i-j})$ by (S3) and the fact that $s-j\geq 0$.

We next consider the general case, i.e.~the case when $f\in E_i$. 
By (S1), one can find $\{ f_n\} _{n\geq 1} \subset \mathcal C^r(\mathbb T^d)$ such that $\Vert f -f_n\Vert _{E_i} \to 0$ as $n\to \infty$.
By the result in the previous paragraph, $ (\mathbf L f_n)^{(k)}$ exists as an element of a Banach space $\mathcal C^0([-1,1],  E_{i-k})$  for all $1\leq k\leq j$.
On the other hand,  it follows from (S3), (S4), \eqref{eq:1209b} and \eqref{eq:1014d2} that
\[
\sup _{\epsilon \in [-1,1]}\left\Vert (\mathbf L f_n)^{(k)}_\epsilon - (\mathbf L f_m)^{(k)}_\epsilon\right\Vert _{E_{i-k}}
\leq C\sum _{\vert \alpha \vert \leq k} \sup _{\epsilon \in [-1,1]}\left\Vert \widehat J_{k,\alpha ,\epsilon }\right\Vert _{\mathcal C^{r-k}} \left\Vert f_n - f_m\right\Vert _{E_i} \to 0
\]
as $n, m\to \infty$.
In particular,
 $\lim _{n\to \infty} (\mathbf Lf_n)^{(j)} $ exists.
 In a similar manner, we can show that $\mathbf L$ is a bounded operator from $\mathcal C^0([-1,1], E_i)$ to $\mathcal C^j([-1,1], E_{i-j})$, so that 
 $\lim _{n\to \infty} (\mathbf Lf_n)^{(j)} =(\mathbf L f)^{(j)}$ in $\mathcal C^0([-1,1],  E_{i-j})$.
In conclusion,
$\mathbf L f  : \epsilon \mapsto  \mathcal L_{T_\epsilon } f$ is in  $\mathcal C^j([-1,1],E_{i-j})$.
\end{proof}

\subsection{Random Anosov maps}
Let $M$ be a compact, connected $\mathcal{C}^\infty$ Riemannian manifold with dimension $d$.
In this section we consider random dynamical systems consisting of Anosov maps lying in a small $\mathcal{C}^{r+1}(M,M)$-neighbourhood of a fixed, topologically transitive Anosov diffeomorphism $T \in \mathcal{C}^{r+1}(M,M)$ for some $r \ge 1$.
The setting we consider is very similar to that of \cite[Section 4]{dragivcevic2020statistical}, however we shall obtain more general conclusions than those of \cite{dragivcevic2020statistical}.
For the remainder of this section we shall fix a topologically transitive Anosov diffeomorphism $T \in \mathcal{C}^{r+1}(M,M)$.
Recall that $(\Omega, \mathcal{F}, \mathbb{P})$ is a Lebesgue space and  $\sigma : \Omega \to \Omega$ is a measurably  invertible, ergodic, measure-preseving map.
For every $\eta > 0$ we define
\begin{equation*}
  \mathcal{O}_\eta(T) = \{ S \in \mathcal{C}^{r+1}(M,M) : d_{\mathcal{C}^{r+1}}(S,T) < \eta \}.
\end{equation*}
Recall that if $\eta$ is sufficiently small then $\mathcal{O}_\eta(T) \subset \mathcal{N}^{r+1}(M,M) $  and every $S \in \mathcal{O}_\eta(T)$ is an Anosov diffeomorphism.
A map $\mathcal{T} : \Omega \to \mathcal{C}^{r+1}(M,M)$ will be said to be measurable if it is $(\mathcal{F}, \mathcal{B}_{\mathcal{C}^{r+1}(M,M)})$-measurable.

We consider random dynamical systems induced by measurable maps $\mathcal{T} : \Omega \to \mathcal{O}_\eta(T)$ for some small, fixed $\eta$, over $\sigma$. Our main result for this subsection concerns the stability properties of the equivariant family of probability measures associated to such  systems.
We will formulate our result in the setting developed by Gou{\"e}zel and Liverani in \cite{gouezel2006banach}.
In particular, in \cite{gouezel2006banach} it is shown that when a  topologically transitive Anosov map is smoothly perturbed the SRB measure varies with similar regularity in certain anisotropic Banach spaces.

A small technical comment is required before proceeding: in \cite{gouezel2006banach} the usual metric on $M$ is replaced by an adapted metric for $T$ (which will also be adapted for $S \in \mathcal{O}_\eta(T)$ provided that $\eta$ is sufficiently small); we shall do the same here. 
We denote by $m$ the Riemannian probability measure induced by the adapted metric on $M$.
For each $q \geq  0$, $p \in \N _0$ with $p \le r$ one obtains a space $B_{p,q}(T)$ by taking the completion of $\mathcal{C}^r(M)$ with respect to anisotropic norms\footnote{Actually, the norms in \cite[Section 3]{gouezel2006banach} are defined on the real Banach space $\mathcal{C}^r(M,\R)$. Here we consider the complexification, which is of no consequence.} $\norm{\cdot}_{p,q}$ as defined in \cite[Section 3]{gouezel2006banach}.
Since the map $T$ is fixed we will just write $B_{p,q}$ in place of $B_{p,q}(T)$.
Our main result for this subsection is the following.
\begin{theorem}\label{thm:quenched_response_anosov}
Let $N, p\in \mathbb N$ and $q\geq 0$ satisfying that $p+q < r-N$.
Then there exists $\eta_0 > 0$ such that every measurable $\mathcal{T} : \Omega \to \mathcal{O}_{\eta_0}(T)$ has an equivariant measurable family of Radon probability measures $\{ \mu ^{\mathcal{T} } _{\omega } \} _{\omega \in \Omega}$ and  $h_{\mathcal T} \in L^\infty(\Omega,  B_{p+N,q})$ such that 
  $h_{\mathcal T}(\omega )(g) =\int g d\mu ^{\mathcal{T} } _{\omega }$ for each $g\in \mathcal C^\infty (M)$ and almost every $\omega$.
  In addition, if $\{ \mathcal{T}_\epsilon : \Omega \to \mathcal{O}_{\eta_0}(T)\}_{\epsilon \in [-1,1]}$ is a family of measurable maps such that
   there is a bounded subset $\mathcal K$ of\footnote{Recall that $\mathcal{C}^{r+1}(M,M)$ is a $\mathcal{C}^{r+1}$ Banach manifold and so for $k \le r+1$ we may talk of $\mathcal{C}^k$ curves taking values in $\mathcal{C}^{r+1}(M,M)$.} $\mathcal{C}^N([-1,1],\mathcal{C}^{r+1}(M,M))$
    satisfying that 
   $\epsilon \mapsto \mathcal{T}_\epsilon(\omega)$  lies in $\mathcal K$ for almost every $\omega$, then  the map $\epsilon \mapsto h_{\mathcal{T}_\epsilon }( \omega )$ is in $\mathcal{C}^{N-1}([-1,1], B_{p, q+N})$ for almost every $\omega$.
\end{theorem}

We will use Theorem \ref{thm:quenched_linear_response}   to prove Theorem \ref{thm:quenched_response_anosov},  with help of Proposition \ref{prop:measurable_extension}, Proposition \ref{prop:existence_equivariant_measure} and Corollary \ref{cor:1211}.
Therefore, what we should do is to check (S1)-(S5), (QR0), (QR2) and (QR3) for appropriate Banach spaces.
We start with the basic properties of the $B_{p,q}$ spaces from \cite{gouezel2006banach}.

\begin{enumerate}
 \item 
 By definition of $\Vert \cdot \Vert _{p,q}$, 
it is straightforward to see that
 $\Vert \partial _l f \Vert _{p,q} \leq \Vert f\Vert _{p+1,q-1}$ for each $f\in \mathcal C^{r}(M)$ and $1\leq l\leq d$. Furthermore, $B_{p+1,q-1} \into B_{p,q}$.
    \item \cite[Lemma 2.1]{gouezel2006banach} If $p+q <r$ then the unit ball in $B_{p+1,q-1}$ is relatively compact in $B_{p,q}$.
    \item \cite[Lemma 3.2]{gouezel2006banach} $\Vert uf \Vert _{p,q} \leq C\Vert u\Vert _{\mathcal C^{p+q}} \Vert f\Vert _{p,q}$ for each $u \in \mathcal C^{p+q}(M)$ and $f\in \mathcal C^{r}(M)$. In particular, if $p+q <r$ then $\mathcal{C}^r(M) \into B_{p,q}$ (see also \cite[Remark 4.3]{gouezel2006banach}).
  \item \cite[Proposition 4.1]{gouezel2006banach} We have $B_{p,q} \into (\mathcal C^q(M) )^*$. Specifically, for each $h \in \mathcal{C}^r(M)$ one obtains a distribution $\tilde{h} \in (\mathcal C^q(M) )^*$ defined by $\tilde{h}(g) = \intf h g dm$.
  The map $h \mapsto \tilde{h}$ continuously extends from $\mathcal{C}^r(M)$ to $B_{p,q}$ and yields the required inclusion.
\end{enumerate}
We also remark that  there exists injections $B_{p,q} \to B_{p-1, q}$ and $B_{p,q} \to B_{p,q'}$ for $q' > q$ due to \cite[Remark 4.2]{gouezel2006banach}.
By the fourth item of the above list, the functional $h \mapsto \intf h dm$ on $\mathcal{C}^r(M)$ extends to a continuous functional on $B_{p,q}$, which we shall also denote by $m$.
The following result summarises some facts from \cite{gouezel2006banach} and \cite{conze2007limit} pertaining to the boundedness and mixing of the Perron-Frobenius operator associated to maps in $\mathcal{O}_\eta(T)$ for $\eta$ small.
We refer the reader to \cite[Lemma 2.2]{gouezel2006banach} and the discussion at the beginning of \cite[Section 7]{gouezel2006banach} for the first and second items, and \cite[Proposition 2.10]{conze2007limit}  for the third item (see also \cite[Section 3]{dragivcevic2020spectral}).
\begin{proposition}\label{prop:gl_seq}
  There exists $0< \eta _0\leq \eta$ such that for any $p \in \N _0$ and $q \geq 0$ with $p+q < r$ we have:
  \begin{enumerate}
    \item For every sequence $\{T_i\}_{i \in \N} \subseteq \mathcal{O}_{\eta _0}(T)$ and $n \in \N$ we have
    \begin{equation*}
      \norm{\LL_{T_{n}} \circ \cdots \circ \LL_{T_{1}}}_{L(B_{p,q})} \le C_{p,q}.
    \end{equation*}
    \item 
    There exists $\alpha_{p,q} \in [0,1)$ such that for every $\{T_i\}_{i \in \N} \subseteq \mathcal{O}_{\eta _0}(T)$, $n \in \N$ and $f \in B_{p+1,q}$ we have
    \begin{equation*}
      \norm{(\LL_{T_{n}} \circ \cdots \circ \LL_{T_{1}}) f }_{p+1,q} \le C_{p,q} \alpha_{p,q}^n \norm{f }_{B_{p+1,q}} + C_{p,q} \norm{f }_{B_{p,q+1}}.
    \end{equation*}
    \item   There exists a constant $\lambda_{p,q} \in [0,1)$ such that for every sequence $\{T_i\}_{i \in \Z} \subseteq \mathcal{O}_{\eta_2}(T)$ and $n \in \N$ we have
  \begin{equation*}
    \norm{\restr{\LL_{T_{n}} \circ \cdots \circ \LL_{T_{1}}}{V_{p,q}}}_{L(B_{p,q})} \le C_{p,q}\lambda_{p,q}^n,
  \end{equation*}
  where $  V_{p,q} = \ker\left(\restr{m}{B_{p,q}}\right) = \left\{h \in B_{p,q} \mid m(h) = 0 \right\}$.
  \end{enumerate}
\end{proposition}

Fix $q \geq 0$ and $p \in \mathbb N$ with $p+q < r-N$.
Let $E_{j} = B_{p+ j, q + N-j}$, $j \in \{0, \ldots, N\}$.
Then, the conditions (S1)-(S4) on theses Banach spaces  immediately follow 
from 
 the above list (recall that each $E_j$ is the completion of $\mathcal C^r(M)$ with respect to $\Vert \cdot \Vert _{E_j}$).
 Furthermore, fix  a bounded subset $\mathcal K$ of $\mathcal{C}^N([-1,1],\mathcal{C}^{r+1}(M,M))$ and let
$\{ \mathcal{T}_\epsilon : \Omega \to \mathcal{O}_{\eta_0}(T)\}_{\epsilon \in [-1,1]}$ be a family of measurable maps such that
   $\epsilon \mapsto \mathcal{T}_\epsilon(\omega)$  lies in $\mathcal K$ for almost every $\omega$.
Then, by virtue of Proposition \ref{prop:measurable_extension}, the first part of Proposition  \ref{prop:gl_seq}  and the above list, 
 Perron-Frobenius operator cocycles $\{ (\mathcal L_{\mathcal T_\epsilon } , \sigma )\} _{\epsilon \in [-1,1]}$ associated with the random dynamics $\{ (\mathcal T_\epsilon ,\sigma )\} _{\epsilon \in [-1,1]}$ are precisely defined on these Banach spaces
and
they satisfy (S5) and (QR0) with $\xi =m$ (see the remark following (s2)) except the $m$-mixing property.  
In fact, (QR2), (QR3) and the mixing of $\{ ( \LL_\mathcal{T_\epsilon }, \sigma )\}_{\epsilon \in [-1,1]}$ on $E_j$ for $j\in \{ 1, N\}$ are  consequences of each items of Proposition \ref{prop:gl_seq}, respectively.
By Corollary \ref{cor:1211}, (QR1), (QR4) and (QR5) also hold for $\{ ( \LL_\mathcal{T_\epsilon }, \sigma )\}_{\epsilon \in [-1,1]}$ on $E_j$
(see also \cite[Lemma 7.1]{gouezel2006banach} and
\cite[Section 9]{gouezel2006banach}).

By Proposition \ref{prop:existence_equivariant_measure} (and the remark following it), 
for each $\epsilon \in [-1,1]$,
 $\mathcal{T} _\epsilon$ has an equivariant measurable family of Radon probability measures $\{ \mu ^{\mathcal{T} _\epsilon } _{\omega } \} _{\omega \in \Omega}$ and  $h_{\mathcal T_\epsilon } \in  L^\infty(\Omega,  E_N) 
 = L^\infty(\Omega,  B_{p+N-1,q})$
  such that 
  $h_{\mathcal T_\epsilon }(\omega )(g) =\int g \mathrm{d}\mu ^{\mathcal{T} _\epsilon} _{\omega }$ for each  $g\in \mathcal C^\infty (M)$ and almost every $\omega$.
Furthermore, we apply Theorem \ref{thm:quenched_linear_response} to deduce the claim that $\epsilon \mapsto h_{\mathcal T_\epsilon }(\omega )$ is in $\mathcal{C}^{N-1}([-1,1], E_0) =\mathcal{C}^{N-1}([-1,1], B_{p-1,q+N})$ for almost every $\omega$, which completes the proof of Theorem \ref{thm:quenched_linear_response}.

\subsection{Random U(1) extensions of expanding maps}
In this section, we will apply Theorem \ref{thm:quenched_linear_response} to 
 quenched linear  response problems  for random U(1) extensions of expanding maps.
Let $\mathcal U $ be the set of $\mathcal C ^{\infty}$  endomorphisms $T: \mathbb T^2 \to \mathbb T^2$ on the torus $\mathbb{T} ^2=\mathbb R^2/ \mathbb Z^2$
  of the form 
\begin{equation}\label{eq:unperturbedsystem}
T:\binom{x}{s}\mapsto
\left( \! \! \begin{array}{cc} E(x) & \\ s+\frac{1}{2\pi}\tau(x) & \!\! \mod 1\end{array} \! \! \right),
\end{equation}
where $E: \mathbb S^1\to \mathbb S^1$ is a $\mathcal C^\infty$ orientation-preserving endomorphism
 on the circle $\mathbb S^1 =\mathbb R/ \mathbb Z$
  and 
$\tau: \mathbb S^1 \rightarrow \mathbb{R}$ is a $\mathcal C^\infty$ function 
($T$ is called the {\it U(1) extension} of $E$ over $\tau $).
U(1) extensions of expanding maps   can be seen as a toy model of (piecewise) hyperbolic flows such as geodesic flows on negatively curved manifolds or dispersive billiard flows (via suspension flows of hyperbolic maps; see \cite{Hasselblatt2017,PP1990}), 
and has been intensively studied by several authors (see e.g.~\cite{Dolgopyat2002,faure2011semiclassical,nakano2015spectra,nakano2016partial,BE2017,FW2017,chen2018spectral}).
When we want to emphasise the dependence of $E$ and $\tau$ in \eqref{eq:unperturbedsystem} on $T$, we write them as $E_T$ and $\tau _T$. 
Fix $T\in \mathcal U$ and 
assume that $E$ is  an expanding map on $\mathbb S^1$ in the sense that
$\min_{x\in \mathbb S^1} E'(x)>1$.
Let $r$ be a positive integer.
For every $\eta >0$ we define 
\[
\mathcal O_\eta (T) =\{ S \in \mathcal U  \mid  d_{\mathcal{C}^{r+1 }}(S,T) < \eta \}.
\]
Note that $\mathcal U\subset \mathcal N^{r+1}(\mathbb T^2,\mathbb T^2)$,
 and that if $\eta$ is sufficiently small then  
 $E_S$ is an expanding map for every $S \in \mathcal{O}_\eta(T)$.

Recall that $(\Omega ,  \mathcal F, \mathbb P)$ is a Lebesgue space and $\sigma :\Omega \to \Omega$ is a measurably invertible, ergodic, $\mathbb P$-preserving map on $(\Omega ,  \mathcal F, \mathbb P)$.
When $\tau _T(x) =\alpha$ for any $x\in \mathbb S^1$ with some constant $\alpha$, then obviously $T$ does not admit any mixing physical  measure because the rotation $s\mapsto s+\frac{1}{2\pi}\alpha$ mod 1 has no mixing physical measure. 
However,   it is known that if $\tau $ satisfies a generic condition, called the \emph{partial captivity} condition, then $T$ admits a unique absolutely continuous invariant probability measure for which  correlation functions  of $T$ decay exponentially fast (in particular, $T$ is mixing).
The partial captivity condition was first introduced by Faure \cite{faure2011semiclassical} and proven to be  generic in \cite{nakano2016partial}.
Furthermore, it was shown in \cite[Theorem 1.6]{nakano2015spectra} that if $T$ satisfies the partial captivity condition, then there is an $\eta _0>0$ and an $m_0\in \mathbb N$, only depending on $T$ (see the comment above Proposition \ref{prop:1218} for more precise choice of $\eta _0$ and $m_0$), such that if $r\geq m_0$ then for any measurable map $\mathcal T:\Omega \to \mathcal O_{\eta _0}(T)$, the RDS $(\mathcal T, \sigma )$ induced by $\mathcal T$ over $\sigma$  admits 
a unique 
 equivariant measurable family of absolutely continuous probability measures $\{ \mu ^{\mathcal T}_\omega \}_{\omega \in \Omega}$ such that the Radon-Nikodym derivative of $\mu ^{\mathcal T}_{ \omega}$ is  in the usual Sobolev space $H^{r}(\mathbb T^2)$ of regularity $r$ for $\mathbb P$-almost every $\omega$ and that 
 quenched correlation functions of $(\mathcal T,\sigma )$ for  $\{ \mu ^{\mathcal T}_\omega \}_{\omega \in \Omega}$
 decay  exponentially fast. 
 See also \cite{dyatlov2015stochastic,gossart2020flat} for related results.

Assume that $T$ satisfies the transversality condition, and fix such an $\eta _0>0$ and an $m_0\in \mathbb N$. 
Assume also that $r\geq m_0+1$.
The main result in this section is the following.

\begin{theorem}\label{thm:1121}
  Let $N$ be positive integers such that $N\leq r-m_0$.
  If $\{ \mathcal{T}_\epsilon : \Omega \to \mathcal{O}_{\eta _0}(T)\}_{\epsilon \in [-1,1]}$ is a family of measurable maps such that there is a bounded subset $\mathcal K$ of $\mathcal{C}^N([-1,1],\mathcal{C}^{r+1}(\mathbb T^2,\mathbb T^2))$ satisfying that $\epsilon \mapsto \mathcal{T}_\epsilon(\omega)$ lies in $\mathcal K$  for $\mathbb P$-almost every $\omega \in \Omega$, then the map $\epsilon \mapsto \mu^{\mathcal{T}_\epsilon }_{\omega}$ is in $\mathcal{C}^{N-1}([-1,1], H^{r-N}(\mathbb T^2))$ for $\mathbb P$-almost every $\omega \in \Omega$.
\end{theorem}

We recall the  basic properties of the Sobolev spaces $H^m(\mathbb T^2)$ with  regularity $m\in \mathbb N_ 0$.
Recall that $\Vert f\Vert _{H^m} ^2= \sum _{\vert \alpha \vert \leq m} \Vert \partial ^\alpha f \Vert _{L^2}^2$.
\begin{enumerate}
  \item By definition of $\Vert \cdot \Vert _{H^m}$, 
it is straightforward to see that $\Vert \partial _l f \Vert _{H^m} \leq \Vert f\Vert _{H^{m+1}}$ for each $f\in \mathcal C^{m+1}(\mathbb T^2)$ and $1\leq l\leq d$, and $\Vert uf \Vert _{H^m} \leq C\Vert u\Vert _{\mathcal C^{m}} \Vert f\Vert _{H^m}$  for each $u, f \in \mathcal C^{m}(\mathbb T^2)$.
  \item By Kondrachov embedding theorem, $H^{m+1}(\mathbb T^2) \into H^m(\mathbb T^2)$ and the unit ball in $H^{m+1}(\mathbb T^2)$ is relatively compact in $H^m(\mathbb T^2)$.
  \item 
  $\mathcal C^{m'}(\mathbb T^2)$ is dense in $H^{m}(\mathbb T^2)$ for each $m'\geq m$ because $\mathcal C^{m'}(\mathbb T^2)\subset H^{m}(\mathbb T^2)\subset \mathcal C^{m-d/2}(\mathbb T^2)$ by Sobolev embedding theorem.
 \item  By Cauchy-Schwarz inequality, we have $H^m(\mathbb T^2) \into (\mathcal C^0(\mathbb T^2) )^*$ by $h\mapsto \tilde h$ given by $ \tilde h(g) = \intf h g dm$ for $g\in \mathcal C^0(\mathbb T^2)$.
\end{enumerate}

Let  $\lambda _0:=(\inf _{S\in \mathcal O_{\eta _0}} \min _{x\in \mathbb S^1} E_{S}'(x) )^{-1}$, which is less than $1$ by taking $\eta _0$ small if necessary.
 Fix $\lambda  \in ( \lambda _0^{\frac{1}{2}} , 1)$. 
Let $m_0$ be a positive integer such that $\lambda ^{2m_0 +1} >\mathrm{deg} (T)$, where $\mathrm{deg} (T)$ is the degree of $T$.
Let $N\leq r-m_0$ be a positive integer.
By taking $\eta _0$ small if necessary, we assume that $\inf _{S\in \mathcal O_{\eta _0}} \min _{x\in \mathbb S^1} E_{S}'(x)<\lambda $.
 Fix a family of measurable maps $\{ \mathcal{T}_\epsilon : \Omega \to \mathcal{O}_{\eta _0}(T)\}_{\epsilon \in [-1,1]}$  such that there is a bounded subset $\mathcal K$ of $\mathcal{C}^N([-1,1],\mathcal{C}^{r+1}(\mathbb T^2,\mathbb T^2))$ satisfying that $\epsilon \mapsto \mathcal{T}_\epsilon(\omega)$ lies in $\mathcal K$  for $\mathbb P$-almost every $\omega \in \Omega$.
 Let $(\mathcal L_{\mathcal T_\epsilon },\sigma )$ be the Perron-Probenius cocycle induced by $(\mathcal T_\epsilon , \sigma )$.
 Then, it follows from \cite[\S 4]{nakano2015spectra} that
 $\mathcal L_{\mathcal T_\epsilon}$ almost surely extends to a unique, bounded operator on $H^m(\mathbb T^2)$ such that $\omega \mapsto \mathcal L_{\mathcal T_\epsilon}(\omega ) : \Omega \to  L(H^m(\mathbb T^2))$ is strongly measurable for each $\epsilon \in [-1,1]$.
The following estimates were proven in    \cite[Subsections 2.3 and 2.4]{CNW}:
\begin{proposition}\label{prop:1218}
There is a constant $\rho \in (0,1)$ (which may depend on $T$ and $\eta _0$) 
such that for all  $\abs{\epsilon} \le 1$ the following holds: 
\begin{itemize}
\item[$\mathrm{(1)}$] For each $m \geq 0$  and $n\geq 1$,
\[
\esssup_{\omega} \Vert \mathcal L_{\mathcal T_\epsilon }^{(n)} (\omega ) \Vert _{L(H^m(\mathbb T^2))} \le C.
\]
\item[$\mathrm{(2)}$]  For each $m\geq m_0$,
 $n\geq 1$ and $f\in H^m(\mathbb T^2)$,
  $$
  \esssup_{\omega} \Vert \mathcal L_{\mathcal T_\epsilon }^{(n)} (\omega )  f\Vert_{H^{m+1}} 
\le C\lambda  ^{n}\lVert f \rVert_{H^{m+1} }+C\lVert f \rVert_{H^{m}}.
$$
\item[$\mathrm{(3)}$]  For each $m\geq m_0$,
 $n\geq 1$ and $f\in H^m(\mathbb T^2)$ with $\intf _{\mathbb T^2} f dm=0$,
  $$
  \esssup_{\omega} \Vert\mathcal L_{\mathcal T_\epsilon }^{(n)} (\omega )  f\Vert_{H^m} 
\le C\rho  ^{n}\lVert f \rVert_{H^m }.
$$
\end{itemize}
  \end{proposition}

We  now prove Theorem \ref{thm:1121}.
Let $m=r-N$.
For $j \in \{0, \dots, N\}$ set $E_{j} = H^{m+j}(\mathbb T^2)$. 
Then, in the same manner as one in the proof of Theorem \ref{thm:quenched_response_anosov},
we can apply  Theorem \ref{thm:quenched_linear_response},  with help of Proposition \ref{prop:measurable_extension}, Proposition \ref{prop:existence_equivariant_measure} and Corollary \ref{cor:1211}, to deduce the claim that $\epsilon \mapsto \mu ^{\mathcal{T}_\epsilon}_\omega$ is in $\mathcal{C}^{N-1}([-1,1], E_0)=\mathcal{C}^{N-1}([-1,1], H^{r-N}(\mathbb T^2))$ for almost every $\omega$, which completes the proof of Theorem \ref{thm:1121}.

\appendix

\section{Proof of Proposition \ref{prop:cont_strong_op_topology}}\label{a:pt}

Let $T\in\mathcal N^{r+1}(M,M)$.
Then it follows from \cite[Corollary 1]{griffiths1963complete} 
 that 
  $T$ is a covering map.
Hence by a basic property of covering spaces, 
there is a discrete topological space $\Gamma$ such that 
for every $x \in M$  there is a neighborhood $U_x$ of $x$ such that $T ^{-1}(\{ x\})$ is homeomorphic to $\Gamma$ and   $T ^{-1}(U_x)$ is homeomorphic to $U_x \times  \Gamma$. 
In other words,  $T^{-1}(U_x)$ is a union of disjoint open sets $\{ \widetilde U_{b , x}\} _{j=1}^{B}$ such that  
  $T: \widetilde U_{b,x}\to U_x$ is a homeomorphism for each $b=1, \ldots ,B$, where $B$ is the cardinality of  $\Gamma$.
If $B= \infty$, then 
since   $\vert \det DT\vert $ is bounded uniformly away from $0$ due to the compactness of $M$,  
we have
\[
m (T^{-1}(U_x)) \geq B \cdot \inf _{y\in M}\vert \det DT(y)\vert \,  m(U_x)  = \infty, 
\]
which contradicts to that $m$ is a finite measure.
Hence $B< \infty$.
Furthermore, there is a small neighborhood $\mathcal U $ of $T$ in $\mathcal N^{r+1}(M,M)$  such that for each $S\in \mathcal U$ and  $x\in M$, there are disjoint open sets $ \{ \widetilde U_{b, x }^S\} _{ b=1} ^{ B}$ such that $S : \widetilde U_{b, x}^S \to U _x$ is a $\mathcal C^{r+1}$ diffeomorphism for each $b$ and  that for each $y\in U_x$,
\begin{equation}\label{eq:gx4}
d_{M}\left( \left(S\vert _{\widetilde U_{b, x }^S} \right) ^{-1}(y) ,   \left(T\vert _{\widetilde U_{b, x }} \right) ^{-1}(y) \right) \to 0
\end{equation}
 as $S\to T$ in $\mathcal N^{r+1}(M,M)$.

Since  $M$ is compact, there are a finite subfamily  $\{U _\lambda \}_{ \lambda \in \Lambda }$  (with $\vert \Lambda \vert <\infty$) of the  open covering $\{ U_x \} _{x\in M}$ of $M$.
For each $\lambda \in \Lambda$ there are disjoint open sets $
\{ \widetilde U_{b, \lambda }\} _{ b=1} ^B$ such that 
$T : \widetilde U_{b, \lambda } \to U _\lambda$  is a $\mathcal C^{r+1}$ diffeomorphism  for each $b=1, \ldots , B$.
Notice that for each $\lambda \in \Lambda$, $x\in U_\lambda$ and a complex-valued function $f$ on $M$, it holds that
\begin{equation}\label{eq:gx2}
\sum _{T(y) =x} f(y) = \sum _{b=1}^{B} f \circ \left( T \vert _{\widetilde U_{b, \lambda}}\right) ^{-1}(x). 
\end{equation}
Let $\{ K_\lambda \} _{\lambda \in \Lambda}$ be a closed covering of $M$ such that $K_\lambda \subset U_\lambda$, and
$\{\rho _\lambda \}_{\lambda \in \Lambda}$   a partition of unity of $M$ subordinate to the covering $\{ K _\lambda \}_{ \lambda \in \Lambda  }$
 (that is,    $\rho _\lambda $ is a $\mathcal C ^\infty$ function on $M$ with values in $[0,1] \subset \R $ such that  the support of $\rho _\lambda$   is contained in $K _\lambda $ for each $\lambda \in \Lambda$ and 
$
\sum _{\lambda \in \Lambda } \rho _\lambda   (x) =1$ for each $ x\in M$).
Then, in view of \eqref{eq:gx2} we get that for each $f\in \mathcal C^r(M)$ and $x\in M$,
\[
 \mathcal L_T f (x)= \sum _{\lambda \in \Lambda } \rho _\lambda (x) \cdot \sum _{T(y)=x}\frac{f(y)}{\vert \det DT(y)\vert }
 = \sum _{\lambda \in \Lambda }  \sum _{b=1}^B \rho _\lambda (x) \cdot\frac{f}{\vert \det DT\vert }\circ \left( T \vert _{\widetilde U_{b, \lambda}}\right) ^{-1}(x)
\]
and 
for each $S\in \mathcal U$, 
\begin{align*}
\left\Vert \mathcal L_T f - \mathcal L_S f\right\Vert _{\mathcal C^r} 
\leq \sum _{\lambda \in \Lambda } \sum _{b=1}^{B} \Vert \rho _\lambda \Vert _{\mathcal C^r} \left\Vert  \frac{f}{\vert \det DT\vert }  \circ \left( T \vert _{\widetilde U_{b, \lambda}}\right) ^{-1} - \frac{f}{\vert \det DS\vert }  \circ \left( S \vert _{\widetilde U_{b, \lambda}^S}\right) ^{-1}  \right\Vert _{\mathcal C^r (K_\lambda )}.
\end{align*}
Therefore, since both $\vert \Lambda \vert$ and $B$ are finite and $  \left\Vert  \vert \det DT\vert ^{-1}  -  \vert \det DS\vert^{-1} \right\Vert _{\mathcal C^r } \to 0$ as $S\to T$  in $\mathcal N^{r+1}(M,M)$, 
 it suffices to show that for each $\lambda \in \Lambda $,  $b=1,\ldots , B$ and $f\in \mathcal C^r(M)$,
\begin{equation}\label{eq:gx3}
 \left\Vert f \circ \left( T \vert _{\widetilde U_{b, \lambda}}\right) ^{-1} - f  \circ \left( S \vert _{\widetilde U_{b, \lambda}^S}\right) ^{-1}  \right\Vert _{\mathcal C^r (K_\lambda )} \to 0 \quad \text{as $S\to T$  in $\mathcal N^{r+1}(M,M)$}.
\end{equation}

Fix $\lambda \in \Lambda $,  $b=1,\ldots , B$ and $f\in \mathcal C^r(M)$. 
By taking $K_\lambda$ small if necessary, we can assume that $K_\lambda$ is included in a local chart of $M$, so we assume that $K_\lambda$ is a closed subset of $\mathbb R^d$ with the dimension $d$ of $M$.
We use the notations $\partial _i(\cdot )$, $\partial ^\alpha (\cdot )$ and $\mathrm{adj}(\cdot )$ given in the proof of Proposition \ref{prop:1209a}.  
Recall that  $ T \vert _{\widetilde U_{b, \lambda}} :  \widetilde U_{b, \lambda} \to  U_{ \lambda}$ is a 
$\mathcal C^{r+1}$ diffeomorphism, so by the inverse function theorem and the fact that $A^{-1} =(\det A )^{-1} \mathrm{adj} (A)$ for any invertible matrix $A$, we have
\[
D\left(\left(T \vert _{\widetilde U_{b, \lambda}} \right) ^{-1}\right)(x) =  D T(y)
 = \left(\det DT(y)\right)^{-1} \mathrm{adj} (DT(y))
\]
with $y=\left(T \vert _{\widetilde U_{b, \lambda}} \right) ^{-1} (x)$ for any $x\in U_{ \lambda}$. 
Since each entry of $\mathrm{adj} (DT(y))$ (the transpose of the cofactor matrix of $DT(y)$) is a polynomial of $\partial _i T$ ($1\leq i\leq d$),
by  the chain rule for derivatives 
we conclude that for each $i=1, \ldots , d$,
\[
\partial _i \left(f\circ  \left(T \vert _{\widetilde U_{b, \lambda}} \right) ^{-1} \right) = \sum _{l=1}^d \left(J_{l} \cdot \partial _l f\right) 
\circ  \left(T \vert _{\widetilde U_{b, \lambda}} \right) ^{-1}\quad \text{on $U_\lambda$,}
\]
where $J_l$ is   a polynomial function of  $\partial _i  T_j$ ($1\leq i,j \leq d$) and $( \det DT)^{-1}$.
Applying this formula repeatedly, we get that  for each 
  multi-index $\alpha $ with $\vert \alpha\vert \leq r$, 
\[
\partial ^\alpha  \left(f\circ  \left(T \vert _{\widetilde U_{b, \lambda}} \right) ^{-1} \right) = \sum _{\vert \beta \vert \leq \vert \alpha \vert} \left(J_{\alpha , \beta} \cdot \partial ^\beta f\right) 
\circ  \left(T \vert _{\widetilde U_{b, \lambda}} \right) ^{-1}\quad \text{on $U_\lambda$,}
\]
where $ J_{\alpha ,\beta } = J_{\alpha,\beta}^T$  is a polynomial function of  $\partial ^\gamma   T_j$ ($1\leq j \leq d$,  $\vert \gamma \vert \leq \vert \beta \vert$) and $( \det DT)^{-1}$.
Now \eqref{eq:gx3} immediately follows from \eqref{eq:gx4}, and this completes the proof.

\section*{Acknowledgments}
Y.N. was partially supported by JSPS KAKENHI Grant Numbers 19K14575 and 19K21834.  H.C. is supported by an Australian Government Research Training Program Scholarship,  a UNSW Science PhD Writing Scholarship, and by the UNSW School of Mathematics and Statistics.

\bibliographystyle{siam}
\bibliography{bibliography}

\end{document}